\documentclass[a4paper,11pt]{article}
\usepackage{amsfonts}
\usepackage{bmpsize}
\usepackage{graphicx}
\usepackage{epstopdf}
\usepackage{amsthm}
\usepackage{amsmath}
\numberwithin{equation}{section}

\ifpdf
  \DeclareGraphicsExtensions{.eps,.pdf,.png,.jpg}
\else
  \DeclareGraphicsExtensions{.eps}
\fi

\usepackage{color,url}
\usepackage{enumitem}
\usepackage{epic}
\usepackage{amssymb}
\usepackage{afterpage}
\usepackage{epsfig}
\usepackage{pstricks}
\usepackage{stmaryrd,url}
\usepackage{tikz,amsmath}
\usepackage{ulem}
\usepackage{algorithm,algpseudocode}
\usepackage{footmisc}
\usepackage{bm}
\usepackage{amsfonts}
\usepackage{mathrsfs}
\usepackage[margin=2.5cm]{geometry}
\usepackage{hyperref}

\makeatletter
\newenvironment{breakablealgorithm}
{
	\begin{center}
		\refstepcounter{algorithm}
		\hrule height.8pt depth0pt \kern2pt
		\renewcommand{\caption}[2][\relax]{
			{\raggedright\textbf{\ALG@name~\thealgorithm} ##2\par}%
			\ifx\relax##1\relax 
			\addcontentsline{loa}{algorithm}{\protect\numberline{\thealgorithm}##2}%
			\else 
			\addcontentsline{loa}{algorithm}{\protect\numberline{\thealgorithm}##1}%
			\fi
			\kern2pt\hrule\kern2pt
		}
	}{
		\kern2pt\hrule\relax
	\end{center}
}

\newcommand{\wt}{\widetilde}

\newtheorem{theorem}{Theorem}[section]
\newtheorem{lemma}[theorem]{Lemma}

\title{Elastic wave propagation in curvilinear coordinates with mesh refinement interfaces by a fourth order finite difference method}

\begin{document}
\author{Lu Zhang\thanks{Department of Applied Physics and Applied Mathematics, Columbia University, New York, NY 10027, USA. \href{mailto:lz2784@columbia.edu}{Email: lz2784@columbia.edu}}\and Siyang Wang\thanks{Department of Mathematics and Mathematical Statistics, Umeå University, Umeå, Sweden. \href{mailto:siyang.wang@umu.se}{Email: siyang.wang@umu.se} Previous affiliation: Division of Applied Mathematics, UKK, M\"alardalen University, V\"aster\aa s, 
    Sweden. }\and N. Anders
  Petersson\thanks{Center for Applied Scientific Computing, Lawrence Livermore National Laboratory,
    Livermore, CA 94551, USA. \href{mailto:petersson1@llnl.gov}{Email: petersson1@llnl.gov}}}
\maketitle

\begin{abstract}
We develop a fourth order accurate finite difference method for the three dimensional elastic wave equation in isotropic media with the piecewise smooth material property. In our model, the material property can be discontinuous at curved interfaces. The governing equations are discretized in second order form on curvilinear meshes by using a fourth order finite difference operator satisfying a summation-by-parts property. The method is energy stable and high order accurate. The highlight is that mesh sizes can be chosen according to the velocity structure of the material so that computational efficiency is improved. At the mesh refinement interfaces with hanging nodes, physical interface conditions are imposed by using ghost points and interpolation.  With a fourth order predictor-corrector time integrator, the fully discrete scheme is energy conserving. Numerical experiments are presented to verify the fourth order convergence rate and the energy conserving property. 
\end{abstract}


\textbf{AMS subject }:    65M06, 65M12

\section{Introduction}
Seismic wave propagation has important applications in earthquake simulation, energy resources exploration, and underground motion analysis. In many practical problems, wave motion is governed by the three dimensional (3D) anisotropic elastic wave equations. The layered structure of the Earth gives rise to a piecewise smooth material property with discontinuities at internal interfaces, which are often curved in realistic models. Because of the heterogeneous material property and internal interfaces, the governing equations cannot be solved analytically, and it is necessary to use advanced numerical techniques to solve the seismic wave propagation problem.

When solving hyperbolic partial differential equations (PDEs), for computational efficiency, it is essential that the numerical methods are high order accurate (higher than second order). This is because high order methods have much smaller dispersion error than lower order methods \cite{Hagstrom2012, Kreiss1972}. However, it is challenging to obtain a stable and high order accurate method in the presence of discontinuous material property and non-trivial geometry. 

Traditionally, the governing equations of seismic wave propagation are solved as a first order system, either in velocity-strain or velocity-stress formulation, which consists of nine equations. With the finite difference method, staggered grids are often used for first order systems, and recently the technique has been generalized to staggered curvilinear grids for the wave equation \cite{OReilly2020}. The finite difference method on non-staggered grids has also been developed for seismic wave simulation in 2D \cite{Kozdon2013} and 3D \cite{Duru2016}.

In this paper, we use another approach that discretizes the governing equations in second order form. Comparing with nine PDEs in a first order system, the second order formulation consists of only three PDEs in the displacement variables. In many cases, this could be a more efficient approach in terms of accuracy and memory usage. For spatial discretization, we consider the finite difference operators constructed in \cite{sjogreen2012fourth} that satisfy a summation-by-parts (SBP) principle, which is a discrete analog of the integration-by-parts principle and is an important ingredient to obtain energy stability. The SBP operators in \cite{sjogreen2012fourth} use a ghost point outside each boundary to impose boundary conditions strongly. The ghost point values are obtained by solving a system of linear equations. This can be avoided by imposing boundary conditions in a weak sense \cite{Carpenter1994} with the SBP operators constructed in \cite{Mattsson2012} that do not use any ghost point. The close relationship between these two types of SBP operators is explored in \cite{wang2018fourth}, where it was also shown in test problems that the approach using ghost points has better CFL property.

In the SBP finite difference framework, a multi-block approach is often taken when the material property is discontinuous. That is, the domain is divided into subdomains such that the internal interfaces are aligned with the material discontinuities. Each subdomain has four sides in 2D and six faces in 3D, which can then be mapped to a reference domain, for example, a unit square in 2D and a unit cube in 3D. In each subdomain, material properties are smooth and SBP operators are used independently for the spatial discretization of the governing equations. To patch subdomains together, physical interface conditions are imposed at internal interfaces \cite{Almquist2019,duru2014stable}. It is challenging to derive energy stable interface coupling with high order accuracy. 

In \cite{petersson2015wave}, a fourth order SBP finite difference method was developed to solve the 3D elastic wave equation in heterogeneous smooth media, where topography in non-rectangular domains is resolved by using curvilinear meshes. The main objective of the present paper is to develop a fourth order method that solves the governing equations in piecewise smooth media, where material discontinuities occur at curved interfaces.   This is motivated by the fact that in realistic models, material properties are only piecewise smooth with discontinuities, and it is important to obtain high order accuracy even at the material interfaces. A highlight of our method is that mesh sizes in each subdomain can be chosen according to the velocity structure of the material property. This leads to difficulties in mesh refinement interfaces, but maximizes computational efficiency. In the context of seismic wave propagation, as going deeper in the Earth, the wave speed gets larger and the wavelength gets longer. Correspondingly, in our model, the mesh becomes coarser with increasing depth. In this way, the number of grid points per wavelength can be kept almost the same in the entire domain. In addition, curved interfaces are also useful when the top surface has a very complicated geometry. If only planar interfaces are used \cite{SW4}, the size of the finest mesh block on top must be large to keep small skewness of the grid. With curved interfaces, the size of the finest mesh block can be reduced without increasing the skewness of the grid. 

In \cite{wang2018fourth}, we developed a fourth order finite difference method for the 2D wave equations with mesh refinement interfaces on Cartesian grids. Our current work generalizes to 3D elastic wave equations on curvilinear grids. In a 3D domain, the material interfaces are 2D curved faces. To impose interface conditions on hanging nodes, we construct fourth order interpolation and restriction operators for 2D grid functions. These operators are compatible with the underlying finite difference operators. With a fourth order predictor-corrector time integrator, the fully discrete discretization is energy conserving. 

The rest of the paper is organized as follows. In Sec.~2, we introduce the governing equations in curvilinear coordinates. The spatial discretization is presented in detail in Sec.~3. Particular emphasis is placed on the numerical coupling procedure at curved mesh refinement interfaces. In Sec.~4, we describe the temporal discretization and present the fully discrete scheme. Numerical experiments are presented in Sec.~5 to verify the convergence rate of the proposed scheme and the energy conserving property. We also demonstrate that the mesh refinement interfaces do not introduce spurious wave reflections. Conclusions are drawn in Sec.~6.

\section{The anisotropic elastic wave equation}
We consider the time dependent anisotropic elastic wave equation in a three dimensional domain ${\bf x}\in\Omega$, where ${\bf x} = (x^{(1)},x^{(2)},x^{(3)})^T$ are the Cartesian coordinates. The domain $\Omega$ is partitioned into two subdomains $\Omega^f$ and $\Omega^c$, with an interface $\Gamma = \Omega^f\cap\Omega^c$. The material property is assumed to be smooth in each subdomain, but may be discontinuous at the interface $\Gamma$. Without loss of generality, we may assume that the wave speed is slower in $\Omega^f$ than in $\Omega^c$, which motivates us to use a fine mesh in $\Omega^f$ and a coarse mesh in $\Omega^c$. We further assume that both $\Omega^f$ and $\Omega^c$ have six, possibly curved boundary faces. Denote ${\bf r} = (r^{(1)},r^{(2)},r^{(3)})^T$, the parameter coordinates, and  introduce smooth one-to-one mappings 
\begin{equation}\label{mapping}
{\bf x}= {\bf X}^{f}({\bf r}) :  [0,1]^3 \rightarrow \Omega^{f} \subset \mathbb{R}^3 \ \ \ \mbox{and} \ \ \ {\bf x} = {\bf X}^{c}({\bf r}) :  [0,1]^3 \rightarrow \Omega^{c} \subset \mathbb{R}^3.
\end{equation}
Let the inverse of the mappings in (\ref{mapping}) be ${\bf r} = {\bf R}^f({\bf x})$ with components ${\bf R}^f({\bf x}) = (R^{f,(1)}, R^{f,(2)}, R^{f,(3)})^T$ and ${\bf r} = {\bf R}^c({\bf x})$ with components ${\bf R}^c({\bf x}) = (R^{c,(1)}, R^{c,(2)}, R^{c,(3)})^T$, respectively. Note that we do not compute the components of the inverse mapping ${\bf R}^c$ and ${\bf R}^f$ in this paper, the definitions here are for the convenience of the rest of the contents.
 
We further assume that the interface $\Gamma$ corresponds to $r^{(3)} = 1$ for the coarse domain and $r^{(3)} = 0$ for the fine domain. Then the elastic wave equation in the coarse domain $\Omega^c$ in terms of the displacement vector ${\bf C} = {\bf C}({\bf r}, t)$ can be written in curvilinear coordinates as (see \cite{petersson2015wave})
\begin{align}\label{elastic_curvi}
	\rho^c\frac{\partial^2{\bf C}}{\partial^2 t} = \frac{1}{J^c}\left[\bar{\partial}_1(A_1^c\nabla_r{\bf C}) + \bar{\partial}_2(A_2^c\nabla_r{\bf C}) +\bar{\partial}_3(A_3^c\nabla_r{\bf C}) \right],\ \ \  {\bf r}\in[0,1]^3,\ \ \  t\geq0,
\end{align}
where $\rho^c$ is the density function in the coarse domain $\Omega^c$.  We define
\begin{align*} 
A_k^c\nabla_r{\bf C} = \sum_{j = 1}^3 N_{kj}^c\bar{\partial}_j {\bf C}, \ \ \ k = 1,2,3,
\end{align*}
with $\nabla_r  = (\bar{\partial}_1, \bar{\partial}_2, \bar{\partial}_3)^T$,  $\bar{\partial}_i =\frac{\partial}{\partial r^{(i)}}$, for $i = 1,2,3$ and
\begin{equation}\label{N_definition}
	N_{ij}^c = J^c\sum_{l,k = 1}^3\xi_{li}^cO_l^TZ^cO_k\xi_{kj}^c, \ \ i,j = 1,2,3,
\end{equation}
where, 
\[ O_{1}^T = \left(\begin{array}{cccccc}
1 & 0 & 0 &0 & 0 & 0\\
0 & 0 & 0 &0 & 0 & 1\\
0 & 0 & 0 &0 & 1 & 0\end{array}\right), \ \  O_{2}^T = \left(\begin{array}{cccccc}
0 & 0 & 0 &0 & 0 & 1\\
0 & 1 & 0 &0 & 0 & 0\\
0 & 0 & 0 &1 & 0 & 0\end{array}\right),  \ \ O_{3}^T = \left(\begin{array}{cccccc}
0 & 0 & 0 &0 & 1 & 0\\
0 & 0 & 0 &1 & 0 & 0\\
0 & 0 & 1 &0 & 0 & 0\end{array}\right),\]
$Z^c$ is a $6\times 6$ stiffness matrices which is symmetric and positive definite and $\xi_{kj} = \frac{\partial r^{(j)}}{\partial x^{(k)}}$. Further, Define $M^c_{lk} = O_l^TZ^cO_k$, then $M_{ii}^c$ are also symmetric positive definite and $M_{ij}^c = M_{ji}^T$. In particular, for the isotropic elastic wave equation, we have
\[ M_{11}^c = \left(\begin{array}{ccc}
2\mu^c+\lambda^c & 0 & 0\\
0 & \mu^c & 0\\
0 & 0 & \mu^c\end{array}\right),\ \ \  M_{12}^c = \left(\begin{array}{ccc}
0 & \lambda^c & 0\\
\mu^c & 0 & 0\\
0 & 0 & 0\end{array}\right), \]
\begin{equation}\label{M_definition}
M_{22}^c = \left(\begin{array}{ccc}
\mu^c & 0 & 0\\
0 & 2\mu^c+\lambda^c & 0\\
0 & 0 & \mu^c\end{array}\right),\ \ \ M_{13}^c = \left(\begin{array}{ccc}
0 & 0 & \lambda^c\\
0 & 0 & 0\\
\mu^c & 0 & 0\end{array}\right),
\end{equation}
\[\ M_{33}^c = \left(\begin{array}{ccc}
\mu^c & 0 & 0\\
0 & \mu^c & 0\\
0 & 0 & 2\mu^c+\lambda^c\end{array}\right),\ \ \ M_{23}^c = \left(\begin{array}{ccc}
0 & 0 & 0\\
0 & 0 & \lambda^c\\
0 & \mu^c & 0\end{array}\right),\]
\[ M_{31}^c = (M_{13}^c)^T, \ \ \  M_{32}^c =(M_{23}^c)^T, \ \ \ M_{21}^c =(M_{12}^c)^T.\]
Here, $\lambda^c$ and $\mu^c$ are the first and second Lam{\'{e}} parameters, respectively. 

From (\ref{N_definition}) we find that even in the isotropic case the matrices $N_{ij}^c$ are full. Hence, wave propagation in isotropic media has anisotropic properties in curvilinear coordinates. In both isotropic and anisotropic material, the matrices $N_{ii}^c$, $i = 1,2,3$, are symmetric positive definite and $N_{ij}^c=\big(N_{ji}^c\big)^T$, $i,j=1,2,3$. 

Last, $J^c$ is the Jacobian of the coordinate transformation with
\[J^c = \mbox{det} \left(\bar{\partial}_1 {\bf X}^c, \bar{\partial}_2 {\bf X}^c, \bar{\partial}_3 {\bf X}^c\right)\in (0,\infty).\] 
Denote the unit outward normal ${\bf n}_i^{c,\pm} = (n_i^{c,\pm,1},n_i^{c,\pm,2},n_i^{c,\pm,3})$, $i = 1,2,3$, for the boundaries of the subdomain $\Omega^c$, then
\begin{align}\label{outward_normal}
{\bf n}_i^{c,\pm}  = \pm \frac{\nabla_x R^{c,(i)}}{|\nabla_x R^{c,(i)}|}.
\end{align}
Here, $\nabla_x = (\partial_1, \partial_2, \partial_3)^T$, $\partial_i = \frac{\partial}{\partial
  x^{(i)}}$, $i = 1,2,3$. Here, $'+'$ corresponds to $r^{(i)} = 1$ and $'-'$ corresponds to $r^{(i)}
= 0$. The relation between covariant basis vectors $\bar{\partial}_i {\bf X}^c, i = 1,2,3$ and
contravariant basis vectors $\nabla_x R^{c,(i)}, i = 1,2,3$ can be found in \cite{petersson2015wave,
  thompson1985numerical}. The elastic wave equation in curvilinear coordinates for the fine domain
in terms of the displacement vector ${\bf F} = {\bf F}({\bf r}, t)$ is defined in the same way as in
the coarse domain. We have
\begin{align}\label{elastic_curvi_f}
	\rho^f\frac{\partial^2{\bf F}}{\partial^2 t} &= \frac{1}{J^f}\left[\bar{\partial}_1(A_1^f\nabla_r{\bf F}) + \bar{\partial}_2(A_2^f\nabla_r{\bf F}) + \bar{\partial}_3(A_3^f\nabla_r{\bf F}) \right], \ \ \  {\bf r}\in [0,1]^3,\ \ \  t\geq0.
\end{align}

At the interface $\Gamma$, suitable physical interface conditions are the continuity of the traction vectors and the continuity of the displacement vectors,
\begin{equation}\label{interface_cond}
\frac{A_3^c\nabla_r{\bf C}}{J^c\Lambda^c}= \frac{A_3^f\nabla_r{\bf F}}{J^f\Lambda^f}, \quad {\bf F} = {\bf C},
\end{equation}
where
\begin{equation}\label{lambda_cf}
\Lambda^c = \big|\nabla_x R^{c,(3)}\big| , \ \ \ \ \Lambda^f = \big|\nabla_x R^{f,(3)}\big|. 
\end{equation}
Together with suitable physical boundary conditions, the problem (\ref{elastic_curvi}, \ref{elastic_curvi_f}) is well-posed \cite{duru2014stable, petersson2015wave}.

\section{The spatial discretization}

In this section, we describe the spatial discretization for the problem (\ref{elastic_curvi}, \ref{elastic_curvi_f}, \ref{interface_cond}) and start with the SBP operators for the first and second derivative.

\subsection{SBP operators in $1$D}\label{sec_sbp_1d}
Consider a uniform discretization of the domain $x\in[0,1]$ with the grids,
\[\wt
	{\bf x} = [x_0,x_1,\cdots,x_n,x_{n+1}]^T,\ \  x_i = (i-1)h,\ \ i = 0,1,\cdots,n,n+1,\ \ h = 1/(n-1),\]
where $i = 1,n$ correspond to the grid points at the boundary, and $i = 0,n+1$ are ghost points outside of the physical domain. The  operator $D \approx \frac{\partial }{\partial x}$ is a first derivative SBP operator \cite{Kreiss1974,Strand1994} if 
\begin{equation}\label{first_sbp}
({\bf u}, D{\bf v})_h = -(D{\bf u},{\bf v})_h - u_1v_1 + u_nv_n,
\end{equation}
with a scalar product
\begin{equation}\label{inner_product}
({\bf u},{\bf v})_h = h\sum_{i = 1}^{n}\omega_iu_iv_i.
\end{equation}
Here, $0<\omega_i < \infty $ are the weights of scalar product. The SBP operator $D$ has a centered difference stencil at the grid points away from the boundary and the corresponding weights $\omega_i = 1$. To satisfy the SBP identity (\ref{first_sbp}), the coefficients in $D$ are  modified at a few points near the boundary and the corresponding weights $\omega_i \neq 1$. The operator $D$ does not use any ghost points. To discretize the elastic wave equation, we also need to approximate the second derivative with a variable coefficient $(\gamma(x)u_x)_x$. Here, the known function $\gamma(x)>0$ describes the property of the material. There are two different fourth order accurate SBP operators for the approximation of $(\gamma(x)u_x)_x$. The first one $\wt{G}(\gamma){\bf u} \approx (\gamma(x)u_x)_x $, derived by Sj\"ogreen and Petersson \cite{sjogreen2012fourth}, uses one ghost point outside each boundary, and satisfies the second derivative SBP identity,
\begin{equation}\label{sbp_2nd_1}
({\bf u}, \wt{G}(\gamma){\bf v})_h = -S_\gamma({\bf u},{\bf v})-u_1\gamma_1\wt{\bf b}_1{\bf v} + u_n\gamma_n\wt{\bf b}_n {\bf v}.
\end{equation}
Here, the symmetric positive semi-definite bilinear form $S_\gamma({\bf u},{\bf v}) = (D{\bf u},\gamma D{\bf v})_h + ({\bf u}, P(\gamma){\bf v})_{hr}$ does not use any ghost points, $(\cdot,\cdot)_{hr}$ is a standard discrete scalar $L^2$ inner product. The positive semi-definite operator $P(\gamma)$ is small for smooth grid functions but non-zero for odd-even modes, see \cite{petersson2015wave,sjogreen2012fourth} for details. The operators $\wt{\bf b}_1$ and $\wt{\bf b}_n$ take the form 
\begin{equation}\label{sbp_1st_1}
\wt{\bf b}_1 {\bf v} = \frac{1}{h}\sum_{i=0}^{4} \wt{d}_i v_i, \quad\wt{\bf b}_n {\bf v} = \frac{1}{h}\sum_{i=n-3}^{n+1} \wt{d}_i v_i.
\end{equation}
They are fourth order approximations of the first derivative $v_x$ on the left and right boundaries, respectively. We note that the notation $\wt{G}(\gamma){\bf v}$ implies that the operator $\wt{G}$ uses ${\bf v}$ on all grid points $\wt{\bf x}$, but $\wt{G}(\gamma){\bf v}$ only returns values on the grid ${\bf x}$ without ghost points. Therefore, when writing in matrix form, $\wt{G}$ is a rectangular matrix of size $n$ by $n+2$.

 In \cite{wang2018fourth}, a method was developed to convert the SBP operator $\wt{G}(\gamma)$ to another SBP operator $G(\gamma)$ which does not use any ghost point and satisfy
 \begin{equation}\label{sbp_2nd_2}
 ({\bf u}, G(\gamma){\bf v})_h = -S_\gamma({\bf u},{\bf v})-u_1\gamma_1{\bf b}_1{\bf v} + u_n\gamma_n{\bf b}_n{\bf v},
 \end{equation}
 where $S_\gamma(\cdot,\cdot)$ is symmetric positive semi-definite. 
 Here, ${\bf b}_1$ and ${\bf b}_n$ are also finite difference operators for the first derivative at the boundaries, and are constructed to fourth order accuracy. They take the form
 \begin{equation}\label{sbp_1st_2}
 {\bf b}_1 {\bf v} = \frac{1}{h}\sum_{i=1}^{5} d_i v_i,\quad {\bf b}_n {\bf v} = \frac{1}{h}\sum_{i=n-4}^{n} d_i v_i.
 \end{equation}
   In this case, ${G}(\gamma)$ is square in matrix form. We note that in  \cite{Mattsson2012}, Mattsson constructed a similar SBP operator with a third order approximation of the first derivative at the boundaries.  
 
For the second derivative SBP operators $\wt{G}(\gamma)$ in (\ref{sbp_2nd_1}) and $G(\gamma)$ in (\ref{sbp_2nd_2}), both of them use a fourth order five points centered difference stencil to approximate $(\gamma u_x)_x$ at the interior points away from the boundaries. For the first and the last six grid points close to the boundaries, the operators $G(\gamma)$ and $\wt{G}(\gamma)$ use second order accurate one-sided difference stencils. They are designed to satisfy (\ref{sbp_2nd_2}) and (\ref{sbp_2nd_1}), respectively.

In the following section, we use a combination of two SBP operators, $\wt{G}(\gamma)$ and $G(\gamma)$, to develop a multi-block finite difference discretization for the elastic wave equation. The first SBP operator is $\wt{G}(\gamma)$ with ghost point, and the second SBP operator $G(\gamma)$, converted from $\wt{G}(\gamma)$, does not use ghost point.

\subsection{Semi-discretization of the elastic wave equation}\label{semi_discrete_form}

\begin{figure}[htbp]
	\centering
\includegraphics[width=0.6\textwidth,trim={0.4cm 0.7cm 0.8cm 1.4cm}, clip]{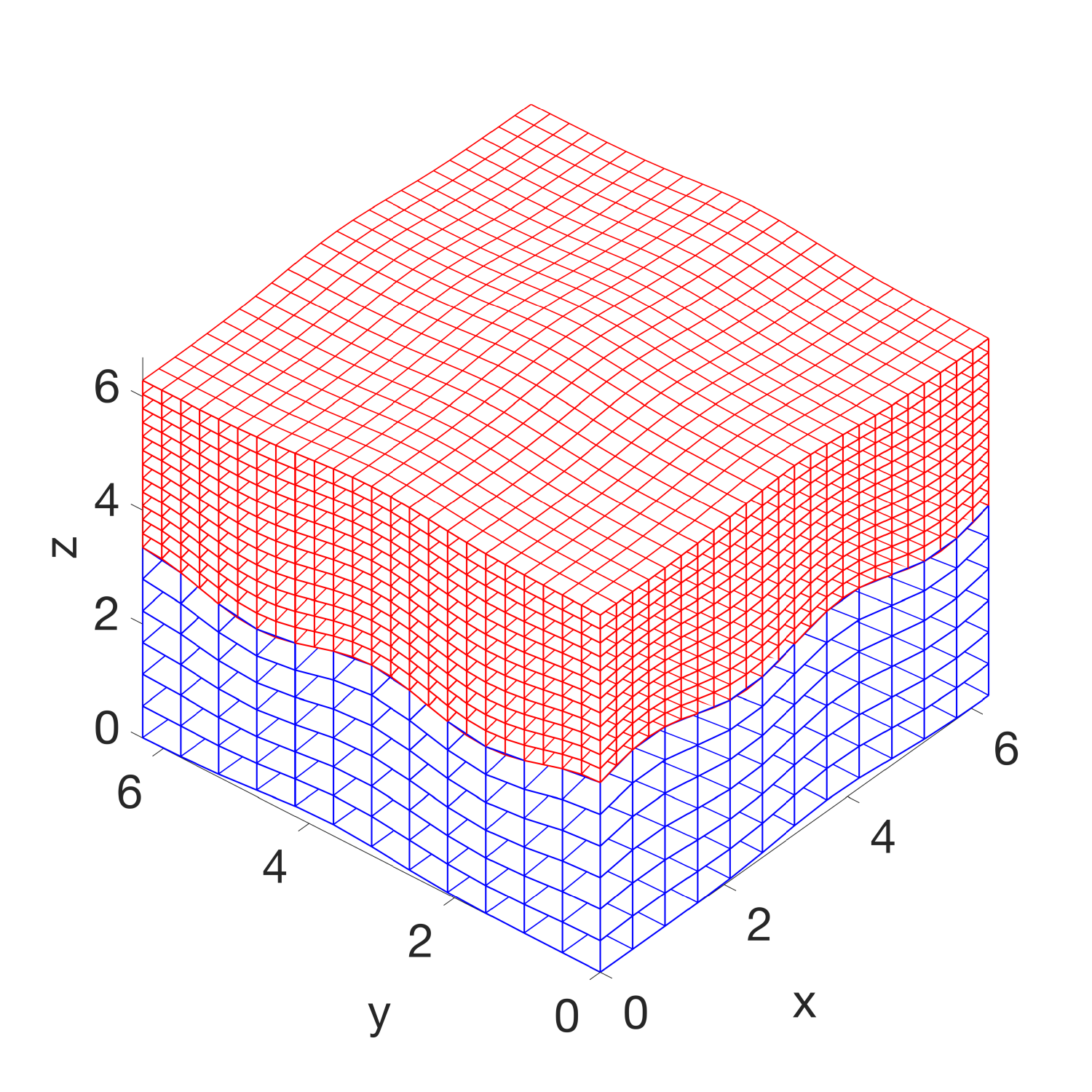}
	\caption{The sketch for the curvilinear mesh of the physical domain $\Omega$. The blue region is the spatial discretization of coarse subdomain $\Omega^c$ and the red region is the spatial discretization of the fine domain $\Omega^f$. Note that $x,y,z$ in the graph correspond to $x^{(1)}, x^{(2)}, x^{(3)}$ respectively. 
	 }\label{physical_discretization}
\end{figure}

In this section, we discretize the elastic wave equations (\ref{elastic_curvi}) and  (\ref{elastic_curvi_f}) with mesh refinement interface $\Gamma$. We assume the ratio of mesh sizes in the reference domains is $1:2$, that is the mesh sizes satisfy
\[h_1(n_1^h-1) = 1, \ \ \ h_2(n_2^h-1) = 1, \ \ \ h_3(n_3^h-1) = 1,\]
and
\[2h_1(n_1^{2h}-1) = 1, \ \ \ 2h_2(n_2^{2h}-1) = 1, \ \ \ 2h_3(n_3^{2h}-1) = 1.\]
 Other ratios can be treated analogously. Figure \ref{physical_discretization} gives an illustration of the discretization of a physical domain. This is an ideal mesh if the wave speed in $\Omega^f$ is half of the wave speed in $\Omega^c$.

In seismic wave simulation, far-field boundary conditions are often imposed in the $x^{(1)}$ and $x^{(2)}$ directions. Here, our focus is on the numerical treatment of the interface conditions (\ref{interface_cond}). Therefore, we assume periodic boundary conditions in $x^{(1)}$ and $x^{(2)}$ and ignore the boundaries in $x^{(3)}$. In Figure \ref{section_discretization}, we fix $x^{(2)} = 0$ and present the $x^{(1)}$-$x^{(3)}$ section of the domain $\Omega$ in both curvilinear space and parameter space.
\begin{figure}[htbp]
	\centering
	\includegraphics[width=0.45\textwidth,trim={1.0cm 2.0cm 1.0cm 1.8cm}, clip]{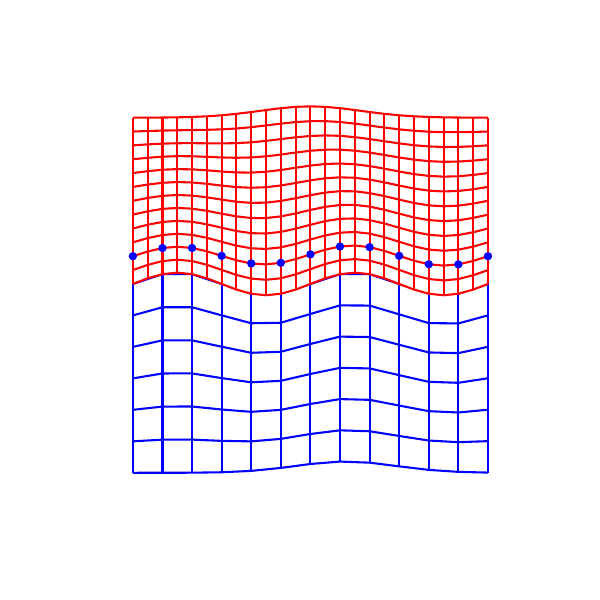}
	\includegraphics[width=0.45\textwidth,trim={1.0cm 2.0cm 1cm 1.8cm}, clip]{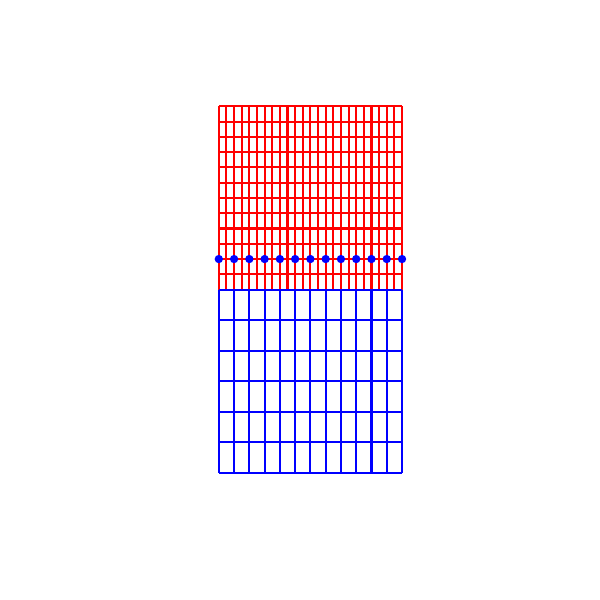}
\caption{The meshes in the physical domain (left) and parameter domain (right) of $x^{(1)}$-$x^{(3)}$ section with $x^{(2)} = 0$. The blue dots are the ghost points for the coarse domain $\Omega^c$.}\label{section_discretization}
\end{figure}
 To condense notations, we introduce the multi-index notations
\[{\bf i} = (i,j,k),\ \ {\bf r}_{\bf i} = (r^{(1)}_i,r^{(2)}_j,r^{(3)}_k),\ \ {\bf x}_{\bf i} = (x^{(1)}_i,x^{(2)}_j,x^{(3)}_k),\]
and group different sets of grid points according to
\begin{equation*}
\begin{aligned}
	I_{\Omega^c} &= \{i = 1,2,\cdots,n_1^{2h}, j = 1,2,\cdots,n_2^{2h}, k = 1,2,\cdots,n_3^{2h}\},\\
	I_{\Gamma^c} & = \{i = 1,2,\cdots,n_1^{2h}, j = 1,2,\cdots,n_2^{2h}, k = n_3^{2h}\},\\
	I_{\Omega^f} &= \{i = 1,2,\cdots,n_1^h, j = 1,2,\cdots,n_2^h, k = 1,2,\cdots,n_3^h\},\\
	I_{\Gamma^f} & = \{i = 1,2,\cdots,n_1^{h}, j = 1,2,\cdots,n_2^{h}, k = 1\}.
\end{aligned}	
\end{equation*}
The physical coordinates of the coarse grid points and fine grid points follow from the mappings ${\bf x}_{\bf i} = {\bf X}^c({\bf r}_{\bf i})$ and ${\bf x}_{\bf i} = {\bf X}^f({\bf r}_{\bf i})$, respectively. We denote a grid function by
\[{\bf u}_{\bf i} = {\bf u}_{i,j,k} = {\bf u}({\bf x}_{\bf i}),\]
where ${\bf u}$ can be either a scalar or vector. To distinguish between the continuous variables and the corresponding approximations on the grid, we use ${\bf c}_{\bf i}$ and ${\bf f}_{\bf i}$ to denote the grid functions for the approximations of ${\bf C}({\bf x}_{\bf i})$ and ${\bf F}({\bf x}_{\bf i})$, respectively. Let {\bf c} and {\bf f} be the vector representations of the grid functions ${\bf c_i}$ and ${\bf f_i}$ respectively. The elements of $\bf c$ and $\bf f$ are ordered in the following way:\\
		a). for each grid point ${\bf x}_{\bf i}$, there is a $3\times 1$ vector, say ${\bf c}_{\bf i} = (c^{(1)}_{\bf i}, c^{(2)}_{\bf i}, c^{(3)}_{\bf i})^T$ and ${\bf f}_{\bf i} = (f^{(1)}_{\bf i}, f^{(2)}_{\bf i}, f^{(3)}_{\bf i})^T$;\\
		b). the grid points are ordered such that they first loop over $r^{(1)}$ direction ($i$), then $r^{(2)}$ direction ($j$), and finally $r^{(3)}$ direction ($k$) as  
		\[{\bf c} = [c^{(1)}_{1,1,1}, c^{(2)}_{1,1,1}, c^{(3)}_{1,1,1},c^{(1)}_{2,1,1}, c^{(2)}_{2,1,1}, c^{(3)}_{2,1,1},\cdots]^T, \quad {\bf f} = [f^{(1)}_{1,1,1}, f^{(2)}_{1,1,1}, f^{(3)}_{1,1,1},f^{(1)}_{2,1,1}, f^{(2)}_{2,1,1}, f^{(3)}_{2,1,1},\cdots]^T.\] 
		We note that $\bf c$ contains the ghost point values for $k = n_3^{2h}+1$, but ${\bf f}$ does not contain any ghost point values.
	
In the spatial discretization, we only use ghost points in the coarse domain and do not use ghost points in the fine domain. Comparing with the traditional approach of using ghost points in both domains, the system of linear equations at the interface becomes smaller and has a better structure. For the rest of the paper, the $\sim$ over an operator represents that the operator applies to a grid function with ghost points. We approximate the elastic wave equation (\ref{elastic_curvi}) in $\Omega^c$ by
{
\begin{equation}\label{elastic_semi_c}
\left(({\rho}^{2h}\otimes{\bf I})(J^{2h}\otimes {\bf I})\frac{d^2{{\bf c}}}{dt^2}\right)_{\bf i} = \wt{\mathcal{L}}^{2h}_{\bf i} {{\bf c}},\quad {\bf i}\in I_{\Omega^c},\quad t>0,
\end{equation}
where $\rho^{2h}$ and $J^{2h}$ are $n_1^{2h}n_2^{2h}n_3^{2h}\times n_1^{2h}n_2^{2h}n_3^{2h}$ diagonal matrices with the diagonal elements $\rho^{2h}_{\bf i} = \rho^c({\bf x}_{\bf i})$ and $J^{2h}_{\bf i} = J^c({\bf x}_{\bf i})$, ${\bf i}\in I_{\Omega^c}$; the matrix ${\bf I}$ is a $3\times 3$ identity matrix because the spatial dimension of the governing equation is $3$; finally, the discrete spatial operator is
\begin{equation}\label{L_operator}
\wt{\mathcal{L}}^{2h} {{\bf c}} = \left(\sum_{l=1}^2{Q}_l^{2h}({N}_{ll}^{2h}){\bf c}+\wt{{G}}_3^{2h}({N}_{33}^{2h}){{\bf c}}+\sum_{l=1}^3\sum_{m=1,m\neq l}^3{D}_l^{2h}({N}_{lm}^{2h}{D}_m^{2h}{\bf c})\right),
\end{equation}
which uses ghost points when computing $\wt{G}_3^{2h}(N^{2h}_{33}){\bf c}$. 
In Appendix \ref{appendix_cdomain}, the terms $Q_l^{2h}(N_{ll}^{2h}){\bf c}$, $\wt{G}_3^{2h}(N_{33}^{2h}){\bf c}$ and $D_{l}^{2h}(N_{lm}^{2h}D_m^{2h}{\bf c})$ are presented, which approximate $\bar{\partial}_l(N_{ll}\bar{\partial}_l{\bf C})$, $\bar{\partial}_3(N_{33}\bar{\partial}_3 {\bf C})$ and $\bar{\partial}_l(N_{lm}\bar{\partial}_m {\bf C})$, respectively.

Next, we approximate the elastic wave equation (\ref{elastic_curvi_f}) on the fine grid points. For all fine grid points that are not located at the interface $\Gamma$, the semi-discretization  is
{
\begin{equation}\label{elastic_semi_f}
\left(({\rho}^{h}\otimes{\bf I})(J^h\otimes{\bf I})\frac{d^2{{\bf f}}}{dt^2}\right)_{\bf i} = {\mathcal{L}}^{h}_{\bf i} {{\bf f}},\quad {\bf i}\in I_{\Omega^f}\backslash I_{{\Gamma^f}},\quad t>0.
\end{equation}
Here, $\rho^{h}$ and $J^{h}$ are $n_1^{h}n_2^{h}n_3^{h}\times n_1^{h}n_2^{h}n_3^{h}$ diagonal matrices with the diagonal elements $\rho^h_{\bf i} = \rho^f({\bf x}_{\bf i})$ and $J^h_{\bf i} = J^f({\bf x}_{\bf i})$, ${\bf i}\in I_{\Omega^f}$. And the discrete spatial operator is
\begin{equation}\label{Lf_operator}
{\mathcal{L}}^{h} {{\bf f}} = \left(\sum_{l=1}^2{Q}_l^{h}({N}_{ll}^h){\bf f}+{G}_3^{h}({N}_{33}^h){\bf f}+\sum_{l=1}^3\sum_{m=1,m\neq l}^3{D}_l^{h}({N}_{lm}^{h}{D}_m^{h}{\bf f})\right).
\end{equation}
Here, the term ${G}_3^{h}(N_{33}^{h}){\bf f}$ approximating $\bar{\partial}_3(N_{33}\bar{\partial}_3 {\bf F})$ without using any ghost points is presented in Appendix \ref{appendix_cdomain}; the terms $Q_l^{h}(N_{ll}^{h}){\bf f}$ and $D_{l}^{h}(N_{lm}^{2h}D_m^{h}{\bf f})$ are defined similar as those in (\ref{L_operator}) and are used to approximate $\bar{\partial}_l(N_{ll}\bar{\partial}_l{\bf F})$ and $\bar{\partial}_l(N_{lm}\bar{\partial}_m {\bf F})$, respectively. 

For the approximation at the interface $\Gamma$, we obtain the numerical solution using a scaled interpolation operator
\begin{equation}\label{continuous_sol}
{\bf f}_{\bf i} = {\mathcal{P}}_{\bf i}({\bf c}),\quad {\bf i}\in I_{\Gamma^f},
\end{equation}
which imposes the continuity of the solution at the interface $\Gamma$. 
For energy stability, the operator ${\mathcal{P}}$ must be of a specific form
\begin{equation}\label{scaleP}
{\mathcal{P}} = \left(({J}^h_\Gamma {\Lambda}^h)^{-\frac{1}{2}}{\bf P}({J}^{2h}_\Gamma {\Lambda}^{2h})^{\frac{1}{2}}\right)\otimes{\bf I}.
\end{equation}
Here, $J_{\Gamma}^h$ and $\Lambda^h$ are $n_1^{h}n_2^{h}\times n_1^{h}n_2^{h}$ diagonal matrices with diagonal elements $J_{\Gamma,{\bf i}}^h = J^f({\bf x}_{\bf i})$ and $\Lambda_{\bf i}^h = \Lambda^f({\bf x_i})$, ${\bf i}\in I_{\Gamma^f}$, with $\Lambda^f$ is given in (\ref{lambda_cf}). Similarly, ${J}_{\Gamma}^{2h}$ and ${{\Lambda}^{2h}}$ are $n_1^{2h}n_2^{2h}\times n_1^{2h}n_2^{2h}$ diagonal matrices with diagonal elements $J_{\Gamma,{\bf i}}^{2h} = J^c({\bf x}_{\bf i})$ and $\Lambda_{\bf i}^{2h} = \Lambda^c({\bf x_i})$, ${\bf i}\in I_{\Gamma^c}$, with $\Lambda^c$ is given in (\ref{lambda_cf}). Finally, ${\bf P}$ is an interpolation operator of size $n_1^hn_2^h\times n_1^{2h}n_2^{2h}$ for scalar grid functions at $\Gamma^c$. Since the spatial discretization is fourth order accurate, we also use a fourth order interpolation. With mesh refinement ratio  $1:2$, the stencils ${\bf P}$ have four cases as illustrated in  Figure \ref{interpolation}. Consequently, the scaled interpolation operator $\mathcal{P}$ is also fourth order accurate.

In the implementation of our scheme, we use \eqref{continuous_sol} to obtain the solution at the interface of the fine domain. However, in the energy analysis in  Sec.~\ref{sec_energy}, it is more convenient to use the equivalent form
\begin{equation}\label{elastic_semi_f_i}
\left(({\rho}^h\otimes{\bf I})(J^h\otimes{\bf I}) \frac{d^2{\bf f}}{dt^2} \right)_{\bf i}=
\mathcal{L}^h_{\bf i}{\bf f} + {\bm \eta}_{\bf i}, \quad {\bf i}\in I_{\Gamma^f}
\end{equation}
with 
\begin{equation}\label{eta}
{\bm \eta} = \left(({\rho}^hJ^h)\otimes{\bf I}\right){\mathcal{P}}\left(\left(({\rho^{2h}J^{2h}})\otimes{\bf I}\right)^{-1}\wt{\mathcal{L}}^{2h} {\bf c}\right) - \mathcal{L}^{h}{\bf f}.
\end{equation}
The variable $\bm \eta$ in (\ref{eta}) is approximately zero with a second order truncation error, which is of the same order as the boundary stencil of the SBP operator. Hence,  $\bm \eta$ does not affect the order of truncation error in the spatial discretization. 
For the simplicity of analysis, we introduce a general notation for the schemes (\ref{elastic_semi_f}) and (\ref{elastic_semi_f_i}) in the fine domain $\Omega^f$,
\begin{align}\label{fine_scheme}
\left(({\rho}^h\otimes{\bf I})(J^h\otimes{\bf I})\frac{d^2{\bf f}}{dt^2}\right)_{\bf i} = \hat{\mathcal{L}}^h_{\bf i}{\bf f} = \left\{
\begin{aligned}
&\mathcal{L}^h_{\bf i}{\bf f} +{\bm \eta}_{\bf i}, \quad {\bf i}\in I_{\Gamma^f}\\
&\mathcal{L}^h_{\bf i}{\bf f},\quad\quad\quad {\bf i}\in I_{\Omega^f}\backslash I_{\Gamma^f} 
\end{aligned}
\right. \quad t > 0.
\end{align}

The following condition imposes continuity of traction at the interface, the first equation in (\ref{interface_cond}),
\begin{equation}\label{continuous_traction}
\left(\left((\Lambda^{2h}J_{\Gamma}^{2h})\otimes{\bf I}\right)^{-1}\wt{\mathcal{A}}_3^{2h}{\bf c}\right)_{\bf i}
= {\mathcal{R}}_{\bf i}\Big(\left((\Lambda^hJ_{\Gamma}^h)\otimes{\bf I}\right)^{-1}(\mathcal{A}_3^h{\bf f}-h_3\omega_1{\bm \eta})\Big), \quad {\bf i}\in I_{\Gamma^c}.
\end{equation}
Here,  $\left((\Lambda^{2h}J_{\Gamma}^{2h})\otimes{\bf I}\right)^{-1}\wt{\mathcal{A}}_3^{2h}{\bf c}$ and $\left((\Lambda^hJ_{\Gamma}^h)\otimes{\bf I}\right)^{-1}\mathcal{A}_3^h{\bf f}$ are approximations of the traction at the interface on the coarse grid and fine grid, respectively. The definitions of $\wt{\mathcal{A}}_3^{2h}{\bf c}$ and $\mathcal{A}_3^{h}{\bf f}$ are given in Appendix \ref{appendix_cdomain}. The operator $\mathcal{R}$ is a scaled restriction operator with the structure 
\begin{equation}\label{scaleR}
 {\mathcal{R}} =  \left(({J}^{2h}_\Gamma{\Lambda}^{2h})^{-\frac{1}{2}}{\bf R}({J}^{h}_\Gamma {\Lambda}^h)^{\frac{1}{2}}\right)\otimes {\bf I},
 \end{equation}
 where the stencils of ${\bf R}$  in Figure \ref{restriction} are determined by the compatibility condition ${\bf R}=\frac{1}{4}{\bf P}^T$. It is a restriction operator of size $n_1^{2h}n_2^{2h}\times n_1^hn_2^h$ for scalar grid functions at $\Gamma^f$.   Finally, $h_3\omega_1{\bm \eta}$ in \eqref{continuous_traction} is a term essential for stability, because in the stability analysis in the next section it cancels out ${\bm \eta}$ in the fine domain spatial discretization \eqref{fine_scheme}. The term is smaller than the truncation error of spatial discretization, so it does not affect the overall order of truncation error. Hence, (\ref{continuous_traction})  is a sufficiently accurate approximation for the continuity of traction at the interface.  
 As will be seen later, the compatibility condition, as well as the scaling of the interpolation and restriction operators, are important for energy stability \cite{Lundquist2018}. 
  We also remark that the condition \eqref{continuous_traction} determines the ghost points values in the coarse domain. 

Let ${\bf u}$ and ${\bf v}$ be grid functions in the coarse domain $\Omega^c$. We define the discrete inner product at the interface by
\begin{equation}\label{scalar_product_discrete_interface_c}
\left<{\bf u}, {\bf v}\right>_{2h} = 4h_1h_2\sum_{i=1}^{n_1^{2h}}\sum_{j=1}^{n_2^{2h}}{  J}_{\Gamma,i,j,n_3^{2h}}^{2h}\Lambda_{i,j,n_3^{2h}}^{2h}({\bf u}_{i,j,n_3^{2h}}\cdot {\bf v}_{i,j,n_3^{2h}}).
\end{equation}
 Similarly, the discrete inner product at the interface for fine domain $\Omega^f$ is defined as
\begin{equation}\label{scalar_product_discrete_interface_f}
\left<{\bf u}, {\bf v}\right>_{h} = h_1h_2\sum_{i=1}^{n_1^{h}}\sum_{j=1}^{n_2^{h}}{J}_{\Gamma,i,j,1}^h\Lambda_{i,j,1}^h({\bf u}_{i,j,1}\cdot {\bf v}_{i,j,1})
\end{equation}
when $\bf u$ and $\bf v$ are grid functions in fine domain $\Omega^f$. Then we have the following lemma for the interpolation and restriction operators.
 
 \begin{lemma}\label{lemma1}
 	Let ${\bf c}$ and ${\bf f}$ be grid functions at the interface for coarse domain and fine domain, respectively. Then the interpolation operator $\mathcal{P}$ and the restriction operator $\mathcal{R}$ satisfy
 	\begin{equation}\label{pr_relation}
 	\left<\mathcal{P} {\bf c}, {\bf f}\right>_h = \left<{\bf c}, \mathcal{R}{\bf f}\right>_{2h}
 	\end{equation}
 	if the compatibility condition $\bm{R} = \frac{1}{4}\bm{P}^T$ holds. 
 \end{lemma}
 \begin{proof}
 	From \eqref{scalar_product_discrete_interface_c}--\eqref{scalar_product_discrete_interface_f}, the definition of $\mathcal{P}$ in \eqref{scaleP} and $\mathcal{R}$ in \eqref{scaleR}, we obtain
 	\begin{align*}
 	\left<\mathcal{P}{\bf c}, {\bf f}\right>_h &= h_1h_2\left[\left((J_{\Gamma}^h\Lambda^h)^{\frac{1}{2}} {\bf P}(J_{\Gamma}^{2h}\Lambda^{2h})^{\frac{1}{2}}\otimes {\bf I}\right){\bf c}\right]^T{\bf f}\\
 	& = 4h_1h_2 {\bf c}^T  \left[\left((J_{\Gamma}^{2h}\Lambda^{2h})^{\frac{1}{2}}\frac{1}{4}{\bf P}^T(J_{\Gamma}^h\Lambda^h)^{\frac{1}{2}}\otimes {\bf I}\right) {\bf f}\right]\\
 	& =4h_1h_2 {\bf c}^T  \left[\left((J_{\Gamma}^{2h}\Lambda^{2h})^{\frac{1}{2}}{\bf R}(J_{\Gamma}^h\Lambda^h)^{\frac{1}{2}}\otimes {\bf I}\right) {\bf f}\right] = \left<{\bf c}, \mathcal{R}{\bf f}\right>_{2h}
 	\end{align*}
 \end{proof}

\begin{figure}
	\centering
	\includegraphics[width=0.24\textwidth,trim={1.8cm 0.8cm 1.4cm 1.2cm}, clip]{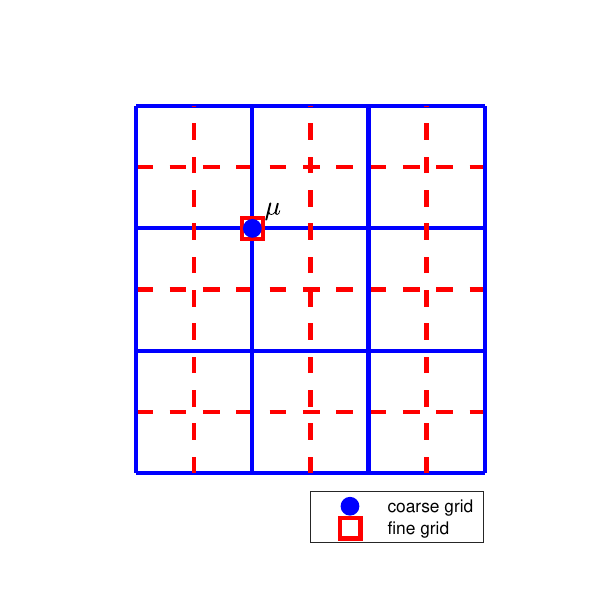}
	\includegraphics[width=0.24\textwidth,trim={1.8cm 0.8cm 1.4cm 1.2cm}, clip]{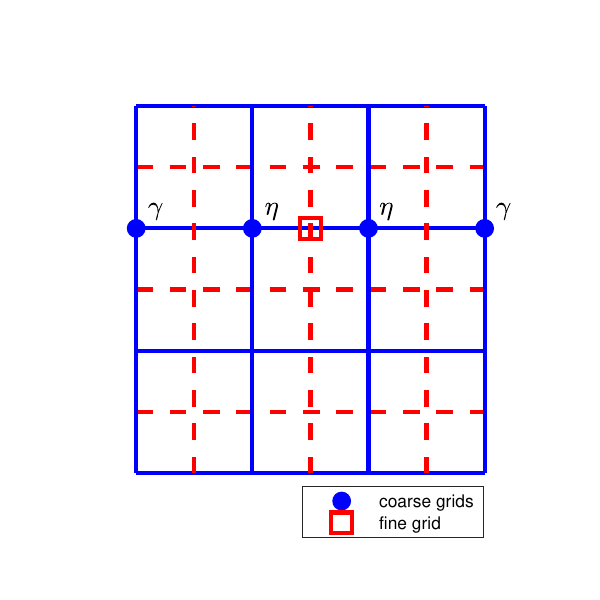}
	\includegraphics[width=0.24\textwidth,trim={1.8cm 0.8cm 1.4cm 1.2cm}, clip]{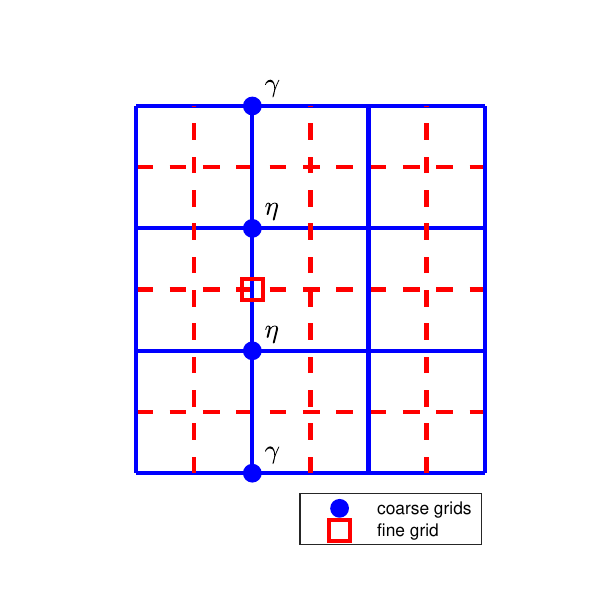}
	\includegraphics[width=0.24\textwidth,trim={1.8cm 0.8cm 1.4cm 1.2cm}, clip]{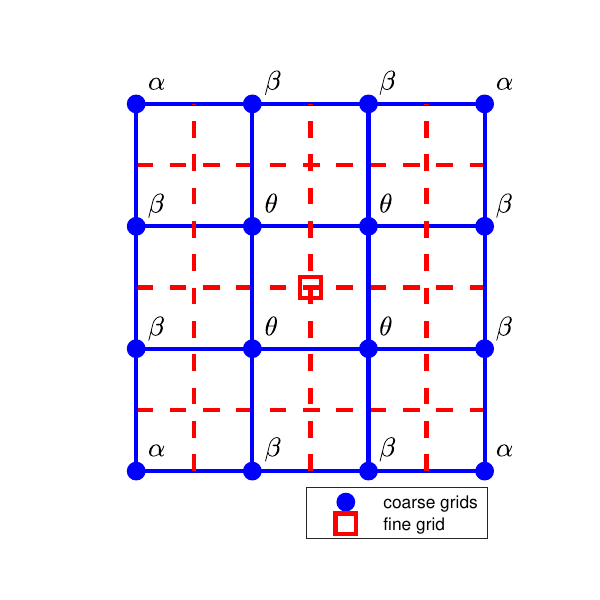}
	\caption{The sketch for the stencils of fourth order interpolation operator ${\bf P}$ in two dimensions with parameters $\gamma = -\frac{1}{16}$, $\eta = \frac{9}{16}$, $\mu = 1$, $\alpha = \frac{1}{256}$, $\beta = -\frac{9}{256}$ and $\theta = \frac{81}{256}$. }\label{interpolation}
\end{figure}
\begin{figure}[htbp]
	\centering
	\includegraphics[width=0.6\textwidth]{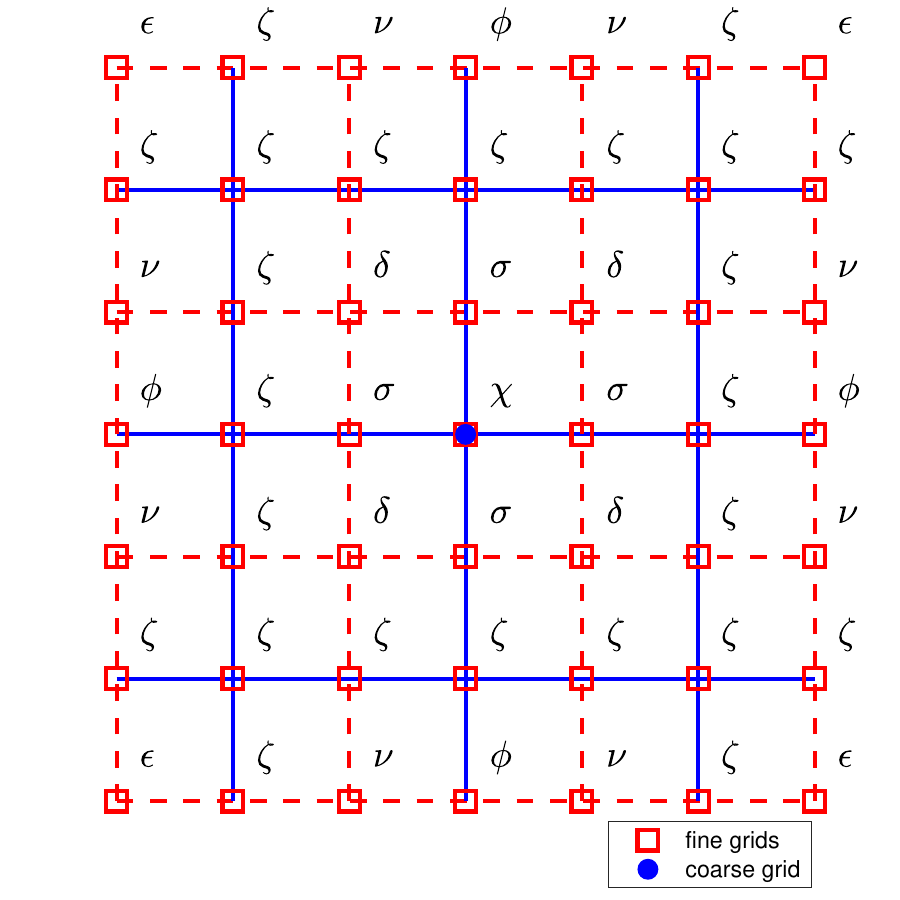}
	\caption{The sketch for the stencil of fourth order restriction operator ${\bf R}$ in two dimensions with parameters $\epsilon = \frac{1}{1024}$, $\nu = -\frac{9}{1024}$, $\phi = -\frac{16}{1024}$, $\delta = \frac{81}{1024}$, $\sigma = \frac{144}{1024}$, $\chi = \frac{256}{1024}$ and $\zeta = 0$.}\label{restriction}
\end{figure}

\subsection{Energy estimate}\label{sec_energy}
In this section, we derive an energy estimate for the semi-discretization (\ref{elastic_semi_c}) and (\ref{fine_scheme}) in Sec.~\ref{semi_discrete_form}. Let ${\bf u}, {\bf v}$ be grid functions in the coarse domain $\Omega^c$ and define the three dimensional discrete scalar product in $\Omega^c$ as
\begin{equation}\label{scalar_product_inner}
({\bf v}, {\bf u})_{2h} = 8h_1h_2h_3\sum_{i=1}^{n_1^{2h}}\sum_{j=1}^{n_2^{2h}}\sum_{k=1}^{n_3^{2h}}\omega_k{J}^{2h}_{i,j,k}({\bf v}_{i,j,k}\cdot {\bf u}_{i,j,k}).
\end{equation}
 Similarly, define the three dimensional discrete scalar product in $\Omega^f$ as
\begin{equation}\label{scalar_product_inner_f}
({\bf v}, {\bf u})_{h} = h_1h_2h_3\sum_{i=1}^{n_1^{h}}\sum_{j=1}^{n_2^{h}}\sum_{k=1}^{n_3^{h}}\omega_k{J}^{h}_{i,j,k}({\bf v}_{i,j,k}\cdot {\bf u}_{i,j,k}),
\end{equation}
where $\bf u$ and $\bf v$ are grid functions in the fine domain $\Omega^f$. Now, we are ready to state the energy estimate of the proposed schemes in Section \ref{semi_discrete_form}. 
\begin{theorem}\label{thm1}
	The semi-discretization (\ref{elastic_semi_c}) and (\ref{fine_scheme}) is energy stable if the interface conditions \eqref{continuous_sol} and \eqref{continuous_traction} are satisfied.
\end{theorem}
\begin{proof}
		Forming the inner product between (\ref{elastic_semi_c}) and $8h_1h_2h_3\omega_k{\bf c}_t$, and summing over $i,j,k$, we have
	\begin{eqnarray}\label{coarse_simple}
	({\bf c}_t, (\rho^{2h}\otimes {\bf I}){\bf c}_{tt})_{2h} = ({\bf c}_t,(J^{2h}\otimes {\bf I})^{-1}\wt{\mathcal{L}}^{2h}{\bf c})_{2h} = -\mathcal{S}_{2h}({\bf c}_t,{\bf c}) + B_{2h}({\bf c}_t,{{\bf c}}),
	\end{eqnarray}
	where $\mathcal{S}_{2h}({\bf c}_t,{\bf c})$ is a symmetric and positive definite bilinear form given in Appendix \ref{appendix_bf}, the boundary term $B_{2h} ({\bf c}_t,{\bf c})$ is given by
	\begin{equation}\label{bounary_c1}
	B_{2h} ({\bf c}_t,{\bf c}) = 4h_1h_2\sum_{{\bf i}\in I_{\Gamma^c}}\frac{d{\bf c}_{\bf i}}{dt}\cdot (\wt{A}_3^{2h}{\bf c})_{\bf i}.
	\end{equation}
	Forming the inner product between (\ref{fine_scheme}) and $h_1h_2h_3\omega_k{\bf f}_t$, and summing over $i,j,k$, we obtain
	\begin{eqnarray}\label{fine_simple}
	({\bf f}_t, (\rho^{h}\otimes {\bf I}){\bf f}_{tt})_{h} \!=\! ({\bf f}_t,(J^{h}\otimes {\bf I})^{-1}\hat{\mathcal{L}}^{h}{\bf f})_{h} \!=\! -\mathcal{S}_{h}({\bf f}_t,{\bf f}) + B_{h}({\bf f}_t,{{\bf f}}) + h_1h_2h_3\omega_1\sum_{{\bf i}\in I_{\Gamma^f}} \frac{d{\bf f}_{\bf i}}{dt}\cdot{\bm \eta}_{\bf i}.
	\end{eqnarray}
Here, $\mathcal{S}_h$ is also a symmetric and positive definite bilinear form given in Appendix \ref{appendix_bf}. The boundary term $B_h ({\bf f}_t,{\bf f})$ has the following form
	\begin{equation}\label{boundary_f1}
	B_h ({\bf f}_t,{\bf f}) = -h_1h_2\sum_{{\bf i}\in I_{\Gamma^f}}\frac{d{\bf f}_{\bf i}}{dt}\cdot(A_3^h {\bf f})_{\bf i}.
	\end{equation}
	 Adding \eqref{coarse_simple} and \eqref{fine_simple} together, we have
	\begin{multline}\label{semi_energy_1}
	\frac{d}{dt}\big[({\bf f}_t,(\rho^h\otimes {\bf I}) {\bf f}_t)_h + \mathcal{S}_{h}({\bf f},{\bf f}) + ({\bf c}_t,(\rho^{2h}\otimes {\bf I}) {\bf c}_t)_{2h} + \mathcal{S}_{2h}({\bf c},{\bf c}) \big]  = \\
	2B_{h}({\bf f}_t,{\bf f}) + 2B_{2h}({\bf c}_t,{\bf c}) +2h_1h_2h_3\omega_1\sum_{{\bf i}\in I_{\Gamma^f}} \frac{d{\bf f}_{\bf i}}{dt}\cdot{\bm \eta}_{\bf i}.
	\end{multline}
	Substituting (\ref{boundary_f1}) and (\ref{bounary_c1}) into (\ref{semi_energy_1}) and combining the definitions of the scalar product at the interface (\ref{scalar_product_discrete_interface_c})--(\ref{scalar_product_discrete_interface_f}), the continuity of solution at the interface \eqref{continuous_sol} and Lemma \ref{lemma1}, we get
	\begin{align*}
	&\hspace{0.4cm}\frac{d}{dt}\left[({\bf f}_t,(\rho^h\otimes {\bf I}) {\bf f}_t)_{h} + \mathcal{S}_{h}({\bf f},{\bf f}) + ({\bf c}_t,(\rho^{2h}\otimes {\bf I}) {\bf c}_t)_{2h} + \mathcal{S}_{2h}({\bf c},{\bf c}) \right]   \\
	& = 2\left<{\bf f}_t,\big(({\Lambda}^{h}{J}^h_\Gamma\big)\otimes {\bf I})^{-1}(-\mathcal{A}_3^h{\bf f}+h_3\omega_1{\bm \eta})\right>_{h}+ 2\left<{\bf c}_t,\big(({\Lambda}^{2h}{J}^{2h}_\Gamma\big)\otimes{\bf I})^{-1}\wt{\mathcal{A}}_3^{2h}{\bf c}\right>_{2h}\\
	& = 2\left<{\mathcal{P}}{\bf c}_t,\big(({\Lambda}^{h}{J}^h_\Gamma)\otimes{\bf I}\big)^{-1}(-\mathcal{A}_3^h{\bf f}+h_3\omega_1{\bm \eta})\right>_{h}+ 2\left<{\bf c}_t, \big(({\Lambda}^{2h}{J}^{2h}_\Gamma)\otimes{\bf I}\big)^{-1}\wt{\mathcal{A}}_3^{2h}{\bf c}\right>_{2h}\\
	& = 2\left<{\bf c}_t,{\mathcal{R}}\Big(\big(({\Lambda}^{h}{J}^h_\Gamma)\otimes{\bf I}\big)^{-1}(-\mathcal{A}_3^h{\bf f}+h_3\omega_1{\bm \eta})\Big)\right>_{2h}+ 2\left<{\bf c}_t,\big(({\Lambda}^{2h}{J}^{2h}_\Gamma)\otimes{\bf I}\big)^{-1}\wt{\mathcal{A}}_3^{2h}{\bf c}\right>_{2h} = 0.
	\end{align*}
Note that the discrete energy for the semi-discretization \eqref{elastic_semi_c} and \eqref{fine_scheme} is given by $({\bf f}_t,(\rho^h\otimes {\bf I}) {\bf f}_t)_{h} + \mathcal{S}_{h}({\bf f},{\bf f}) + ({\bf c}_t,(\rho^{2h} \otimes {\bf I}){\bf c}_t)_{2h} + \mathcal{S}_{2h}({\bf c},{\bf c})$.
\end{proof}

\section{The temporal discretization}
The equations are advanced in time with an explicit fourth order accurate predictor-corrector time integration method. Like all explicit time stepping methods, the time step must not exceed the CFL stability limit. By a similar analysis as in \cite{sjogreen2012fourth}, we require 
\begin{equation*}
\Delta_t\leq C_{\text{cfl}}\min\{h_1,h_2,h_3\}/\sqrt{\kappa_{\max}},
\end{equation*}
where 
$\kappa_{\text{max}}$ is the maximum eigenvalue of the matrices 
\[T_{\bf i}^{\{f,c\}} = \frac{1}{\rho^{\{f,c\}}({\bf r}_{\bf i})}\left(\begin{array}{ccc}
Tr(N_{11}^{\{f,c\}}({\bf r}_{\bf i})) &  Tr(N_{12}^{\{f,c\}}({\bf r}_{\bf i}))& Tr(N_{13}^{\{f,c\}}({\bf r}_{\bf i}))\\
Tr(N_{21}^{\{f,c\}}({\bf r}_{\bf i})) & Tr(N_{22}^{\{f,c\}}({\bf r}_{\bf i})) & Tr(N_{23}^{\{f,c\}}({\bf r}_{\bf i}))\\
Tr(N_{31}^{\{f,c\}}({\bf r}_{\bf i})) & Tr(N_{32}^{\{f,c\}}({\bf r}_{\bf i})) & Tr(N_{33}^{\{f,c\}}({\bf r}_{\bf i}))\end{array}\right), \]
and $Tr(N_{lm}^{\{f,c\}}({\bf r}_{\bf i}))$ represents the trace of $3\times3$ matrix $N_{lm}^{\{f,c\}}({\bf r}_{\bf i})$. Note that $\kappa_{\text{max}}$ is related to the material properties $\mu^{\{f,c\}}, \lambda^{\{f,c\}}$ and $\rho^{\{f,c\}}$. The notation $\{\cdot,\cdot\}$ represents the component-wise identities. We choose the Courant number $C_{\text{cfl}} = 1.3$, which has been shown to work well in practical problems \cite{petersson2015wave,sjogreen2012fourth}. The Courant number shall not be chosen too close to the stability limit so that noticeable reflections at mesh refinement interfaces can be avoided \cite{Collino2003}. In the following, we give detailed procedures about how we apply the fourth order time integrator to the semidiscretizations (\ref{elastic_semi_c}) and  (\ref{fine_scheme}). 

Let ${\bf c}^{n}$ and ${\bf f}^{n}$ denote the numerical approximations of ${\bf C}({\bf x},t_n), {\bf x}\in\Omega^c$ and ${\bf F}({\bf x},t_n), {\bf x}\in\Omega^f$, respectively. Here, $t_n = n\Delta_t, n = 0,1,\cdots$ and $\Delta_t > 0$ is a constant time step. We present the fourth order time integrator with predictor and corrector in  Algorithm \ref{first_alg}.
~\\
\begin{breakablealgorithm}
	\caption{Fourth order accurate time stepping for the semidiscretizations  (\ref{elastic_semi_c}) and  (\ref{fine_scheme}). }\label{first_alg}
	Given $\wt{{\bf c}}^{n}, \wt{{\bf c}}^{n-1}$ and ${\bf f}^{n}, {\bf f}^{n-1}$ that satisfy the discretized interface conditions.
	
	\begin{itemize}
		\item  {Compute the predictor at the interior grid points 
			\[{\bf c}^{*,n+1}_{\bf i} = 2{\bf c}^{n}_{\bf i} - {\bf c}^{n-1}_{\bf i} + \Delta_t^2\left((\rho^{2h}\otimes{\bf I})(J^{2h}\otimes{\bf I})\right)^{-1}{\wt{\mathcal{L}}}^{2h}_{\bf i} {\bf{c}}^{n},\quad {\bf i}\in I_{\Omega^c},\]
			\[{\bf f}^{*,n+1}_{\bf i} = 2{\bf f}^{n}_{\bf i} - {\bf f}^{n-1}_{\bf i} + \Delta_t^2\left((\rho^{h}\otimes{\bf I})(J^{h}\otimes{\bf I})\right)^{-1}\hat{\mathcal{L}}^{h}_{\bf i} {\bf{f}}^{n},\quad {\bf i}\in I_{\Omega^f}\backslash I_{\Gamma^f}.\]
		}
		\item{At the interface $\Gamma$, the values ${\bf f}^{*,n+1}_{\bf i}$  are computed by the continuity of solution 
			\begin{equation*}
			{\bf f}^{*,n+1}_{\bf i} = {\mathcal{P}}_{\bf i}({\bf c}^{*,n+1}),\quad {\bf i}\in I_{\Gamma^f}.
			\end{equation*}
		}
		\item{At the interface $\Gamma$, the ghost point values in $\wt{\bf c}^{*,n+1}$ are computed by solving the equation for the continuity of traction 
				\begin{equation}\label{traction_gamma_pre}
			\left(\left((\Lambda^{2h}J^{2h}_{\Gamma})\otimes{\bf I}\right)^{-1}\wt{\mathcal{A}}_3^{2h}{\bf c}^{\star,n+1}\right)_{\bf i}
			= {\mathcal{R}}_{\bf i}\Big(\left((\Lambda^hJ_{\Gamma}^h)\otimes{\bf I}\right)^{-1}(\mathcal{A}_3^h{\bf f}^{\star,n+1}-h_3\omega_1{\bm \eta}^{\star,n+1})\Big), {\bf i}\in I_{\Gamma^c}.
			\end{equation}
		}
		\item{Evaluate the acceleration at all grid points 
			\begin{equation*}
			{\wt{\bf a}}_c^n= \frac{\wt{\bf c}^{*,n+1}-2\wt{\bf c}^{n}+\wt{\bf c}^{n-1}}{\Delta^2_t},\ \ \ \
			{{\bf a}}_f^{n} = \frac{{\bf f}^{*,n+1}-2{\bf f}^{n}+{\bf f}^{n-1}}{\Delta^2_t}.
			\end{equation*}
		}
		\item{Compute the corrector at the interior grid points
			\[{\bf c}^{n+1}_{\bf i} = {\bf c}^{*,n+1}_{\bf i} + \frac{\Delta_t^4}{12}\left((\rho^{2h}\otimes{\bf I})(J^{2h}\otimes{\bf I})\right)^{-1}\wt{\mathcal{L}}^{2h}_{\bf i}{\bf a}_c^{n},\quad {\bf i}\in I_{\Omega^c},\]
			\[{\bf f}^{n+1}_{\bf i} = {\bf f}^{*,n+1}_{\bf i} + \frac{\Delta_t^4}{12}\left((\rho^{h}\otimes{\bf I})(J^h\otimes{\bf I})\right)^{-1}\hat{\mathcal{L}}^{h}_{\bf i}{\bf a}_f^{n},\quad {\bf i}\in I_\Omega^f.\]
		}
		\item{At the interface $\Gamma$, the values ${\bf f}^{n+1}_{\bf i}$  are computed by the continuity of solution
			\begin{equation*}
			{\bf f}^{n+1}_{\bf i} = {\mathcal{P}}_{\bf i}({\bf c}^{n+1}), \quad {\bf i}\in I_{\Gamma^f}.
			\end{equation*}
		}
		\item{At the interface $\Gamma$, the ghost point values in $\wt{\bf c}^{n+1}$ are computed by solving the equation for the continuity of traction
			\begin{equation}\label{traction_gamma_corr}
			\left(\left((\Lambda^{2h}J^{2h}_{\Gamma})\otimes{\bf I}\right)^{-1}\wt{\mathcal{A}}_3^{2h}{\bf c}^{n+1}\right)_{\bf i}
			= {\mathcal{R}}_{\bf i}\Big(\left((\Lambda^hJ^h_{\Gamma})\otimes{\bf I}\right)^{-1}(\mathcal{A}_3^h{\bf f}^{n+1}-h_3\omega_1{\bm \eta}^{n+1})\Big), \quad {\bf i}\in I_{\Gamma^c}.
			\end{equation}
		}
	\end{itemize}
\end{breakablealgorithm}
~\\

In the Algorithm \ref{first_alg}, we need to solve the linear systems for the continuity of traction at the interface $\Gamma$ in both predictor step (\ref{traction_gamma_pre}) and corrector step (
\ref{traction_gamma_corr}). The linear system matrices of (\ref{traction_gamma_pre}) and (\ref{traction_gamma_corr}) are the same. Therefore, we only present how to solve (\ref{traction_gamma_pre}) in the predictor step.

There are $3n_1^{2h}n_2^{2h}$ unknowns and $3n_1^{2h}n_2^{2h}$ linear equations in (\ref{traction_gamma_pre}). For large problems in three dimensions, it is very memory inefficient to calculate the LU-factorization. Therefore, we use iterative methods to solve the linear system in (\ref{traction_gamma_pre}). In particular, we consider three different iterative methods: the block Jacobi iterative method, the conjugate gradient (CG) iterative method and the preconditioned conjugate gradient iterative method. The detailed methods and a comparison are given in Section \ref{iterative_section}.

\section{Numerical Experiments}
We present four numerical experiments. 
 In Sec.~\ref{convergence_study}, we verify the order of the convergence of the proposed scheme (\ref{elastic_semi_c}, \ref{fine_scheme}, \ref{continuous_sol}, \ref{continuous_traction}).  In  Sec.~\ref{iterative_section}, we present three iterative methods for solving the linear systems (\ref{traction_gamma_pre}) and (\ref{traction_gamma_corr}). The efficiency of the iterative methods is investigated and a comparison with the LU-factorization method is conducted. Next, in Sec.~\ref{gaussian_source} we show that our schemes generate little reflection at the mesh refinement interface. Finally, the energy conservation property is verified in Sec.~\ref{conserved_energy} with heterogeneous and discontinuous material properties.

\subsection{Verification of convergence rate}\label{convergence_study}
We use the method of the manufactured solution to verify the fourth order convergence rate of the proposed scheme. We choose the mapping of the coarse domain $\Omega^c$ as
\[ {\bf x} = {\bf X}^c({\bf r}) = \left(\begin{array}{c}
2\pi r^{(1)} \\
2\pi r^{(2)} \\
r^{(3)}\theta_i\big(r^{(1)},r^{(2)}\big) + (1-r^{(3)})\theta_b\big(r^{(1)},r^{(2)}\big) \end{array}\right), \]
where $0\leq r^{(1)}, r^{(2)}, r^{(3)}\leq 1$, $\theta_i$ represents the interface surface geometry,
\begin{equation}\label{iterface_geometry}
\theta_i\big(r^{(1)},r^{(2)}\big) = \pi+0.2\sin(4\pi r^{(1)})+0.2\cos(4\pi r^{(2)}),
\end{equation}
and $\theta_b$ is the bottom surface geometry,
\begin{equation*}\label{bottom_geometry}
\theta_b\big(r^{(1)},r^{(2)}\big) = 0.2\exp\left(-\frac{(r^{(1)}-0.6)^2}{0.04}\right)+0.2\exp\left(-\frac{(r^{(2)}-0.6)^2}{0.04}\right).
\end{equation*}
As for the fine domain $\Omega^f$, the mapping is chosen to be
\[ {\bf x} = {\bf X}^f({\bf r}) = \left(\begin{array}{c}
2\pi r^{(1)}\\
2\pi r^{(2)}\\
r^{(3)}\theta_t\big(r^{(1)},r^{(2)}\big) + (1-r^{(3)})\theta_i\big(r^{(1)},r^{(2)}\big) \end{array}\right), \]
where $0\leq r^{(1)}, r^{(2)}, r^{(3)}\leq 1$ and $\theta_t$ is the top surface geometry,
\begin{equation*}\label{top_geometry}
\theta_t\big(r^{(1)},r^{(2)}\big) = 2\pi+0.2\exp\left(-\frac{(r^{(1)}-0.5)^2}{0.04}\right)+0.2\exp\left(-\frac{(r^{(2)}-0.5)^2}{0.04}\right).
\end{equation*}
In the entire domain, we choose the density 
\begin{equation*}\label{density_function}
\rho(x^{(1)},x^{(2)},x^{(3)}) = 2 + \sin(x^{(1)}+0.3)\sin(x^{(2)}+0.3)\sin(x^{(3)}-0.2),
\end{equation*}
and material parameters $\mu, \lambda$ 
\begin{equation*}\label{mu_function}
\mu(x^{(1)},x^{(2)},x^{(3)}) = 3 + \sin(3x^{(1)}+0.1)\sin(3x^{(2)}+0.1)\sin(x^{(3)}),
\end{equation*}
and 
\begin{equation*}\label{lambda_function}
\lambda(x^{(1)},x^{(2)},x^{(3)})  = 21+ \cos(x^{(1)}+0.1)\cos(x^{(2)}+0.1)\sin^2(3x^{(3)}).
\end{equation*}
 In addition, we impose a boundary forcing on the top surface and Dirichlet boundary conditions for the other boundaries. The external forcing, top boundary forcing ${\bf g}$ and initial conditions are chosen such that the solutions for both fine domain ($\bf F$) and coarse domain ($\bf C$) are ${\bf F}(\cdot,t) = {\bf C}(\cdot,t) = {\bf u}(\cdot,t) = (u_1(\cdot,t),u_2(\cdot,t),u_3(\cdot,t))^T$ with
\begin{align*}
u_1(\cdot,t) &= \cos(x^{(1)}+0.3)\sin(x^{(2)}+0.3)\sin(x^{(3)}+0.2)\cos(t^2),\\
u_2(\cdot,t) &= \sin(x^{(1)}+0.3)\cos(x^{(2)}+0.3)\sin(x^{(3)}+0.2)\cos(t^2),\\
u_3(\cdot,t) &= \sin(x^{(1)}+0.2)\sin(x^{(2)}+0.2)\cos(x^{(3)}+0.2)\sin(t).
\end{align*}
For example, for the boundary forcing at the top surface, we have 
\begin{equation*}\label{traction_force}
{\bf g} = (g_1,g_2,g_3)^T = \sum_{i=1}^3\left(\sum_{j = 1}^3 M_{ij}^f\frac{\partial{\bf u}}{\partial x^{(j)}}\right) n^{f,+,i}_3,
\end{equation*}
where, $M_{ij}^f$ and $n^{f,+,i}_3$ are defined in (\ref{M_definition}) and (\ref{outward_normal}), respectively.

The problem is evolved until final time $T = 0.5$. In Table \ref{convergence_rate}, we use $L_2$ to represent the $L^2$ error in the entire domain $\Omega = \Omega^c\cup\Omega^f$. The notations $L_2^f$ and $L_2^c$ represent the $L^2$ error in the fine domain $\Omega^f$ and coarse domain $\Omega^c$, respectively. The convergence rates are shown in the parentheses in Table \ref{convergence_rate}. We observe that the convergence rate is fourth order for all cases. Even though the boundary accuracy of the SBP operator is only second order, the optimal convergence rate is fourth order. For a more detailed analysis of the convergence rate, we refer to \cite{Wang2017, Wang2018b}.  To solve the linear system for ghost point values, we use a block Jacobi iterative method. In the following section, we study two more iterative methods and compare them in terms of the condition number and the number of iterations.

\begin{table}[htb]
	\begin{center}
		\begin{tabular}{|c|c c c|}
			\hline
			$2h_1 = 2h_2 = 2h_3 = 2h$   & $L_2$ & $L_2^f$ & $L_2^c$  \\
			\hline
			$2\pi/24$ &2.2227e-03 ~~~~~~~~ & 8.0442e-04 ~~~~~~~~ & 2.0720e-03 ~~~~~~~~\\
			\hline
			$2\pi/48$ &1.4142e-04 (3.97) & 5.1478e-05 (3.97) & 1.3171e-04 (3.98)\\
			\hline 
			$2\pi/96$ &8.6166e-06 (4.04) & 3.0380e-06 (4.08) & 8.0632e-06 (4.03)\\
			\hline
		\end{tabular}
	\end{center}
	\caption{The $L^2$ error and corresponding convergence rates of the fourth order SBP method.}\label{convergence_rate}
\end{table}

\subsection{Iterative methods}\label{iterative_section}
In this section, we use the same example as in Sec.~\ref{convergence_study}. For the proposed scheme (\ref{elastic_semi_c}, \ref{fine_scheme}, \ref{continuous_sol}, \ref{continuous_traction}), we need to solve linear systems with $3n_1^{2h}n_2^{2h}$ unknown ghost point values on the coarse grid. At each time step, two linear systems with the same matrix are solved for the continuity of traction at the interface $\Gamma$. 

We investigate three iterative methods: the block Jacobi method, the conjugate gradient method  and the preconditioned conjugate gradient method. We note that the coefficient matrix of the linear system arising from the continuity of traction at interface $\Gamma$ is not symmetric for this test problem. However, our experiment shows that both the conjugate gradient method and the preconditioned conjugate gradient method converge.

For the problem proposed in Sec.~\ref{convergence_study}, the structure of the coefficient matrix of the linear system arising in (\ref{continuous_traction}) is shown in Figure \ref{Mass_matrix}, which is determined by the interpolation operator ${\mathcal{P}}$ and restriction operator ${\mathcal{R}}$. In this example, we use $n_1^{2h} = n_2^{2h}=13, n_3^{2h} = 7$. We choose the entries indicated by red color in Figure \ref{Mass_matrix} to be the block Jacobi matrix in the block Jacobi iterative method and the preconditioning matrix in the preconditioned conjugate gradient  method. The absolute error tolerance is set to be $10^{-7}$ for all three iterative methods and $h_1 = h_2 = h_3 = h$.

\begin{table}[htbp]
	\begin{center}
		\begin{tabular}{|c|c c c|}
			\hline
			$2h$   & ~~~~ CG ~~~~& Block Jacobi & Preconditioned CG  \\
			\hline
			$2\pi/24$ &37.78& 24.96& 4.01\\
			\hline
			$2\pi/48$ &38.61 & 25.38 & 2.87\\
			\hline 
			$2\pi/96$ &39.14 &25.43 & 2.25\\
			\hline
		\end{tabular}
	\end{center}
		\caption{The condition number of the matrices in the conjugate gradient method, the block Jacobi method and the preconditioned conjugate gradient method.}\label{condition_number}
\end{table} 
Table \ref{condition_number} shows the condition number of the original coefficient matrix, the block Jacobi matrix and the coefficient matrix after applying the preconditioning matrix. We observe that the condition number for preconditioned conjugate gradient method is smallest and is consistent with the results of iteration number for different iterative methods: there are around $44$ iterations for the conjugate gradient method, $13$ iterations for the block Jacobi method and $9$ iterations for the preconditioned conjugate gradient method.

In comparison, we have also performed an LU factorization for the linear system when the mesh size $2h = 2\pi/96$, and the computation takes 40.6 GB memory. In contrast, with the block Jacobi method, the peak memory usage is only 1.2 GB. For large-scale problems, the memory usage becomes infeasible for the LU factorization. 

\begin{figure}[H]
	\centering
	\includegraphics[width=0.45\textwidth,trim={0.6cm 1cm 1cm 1.2cm}, clip]{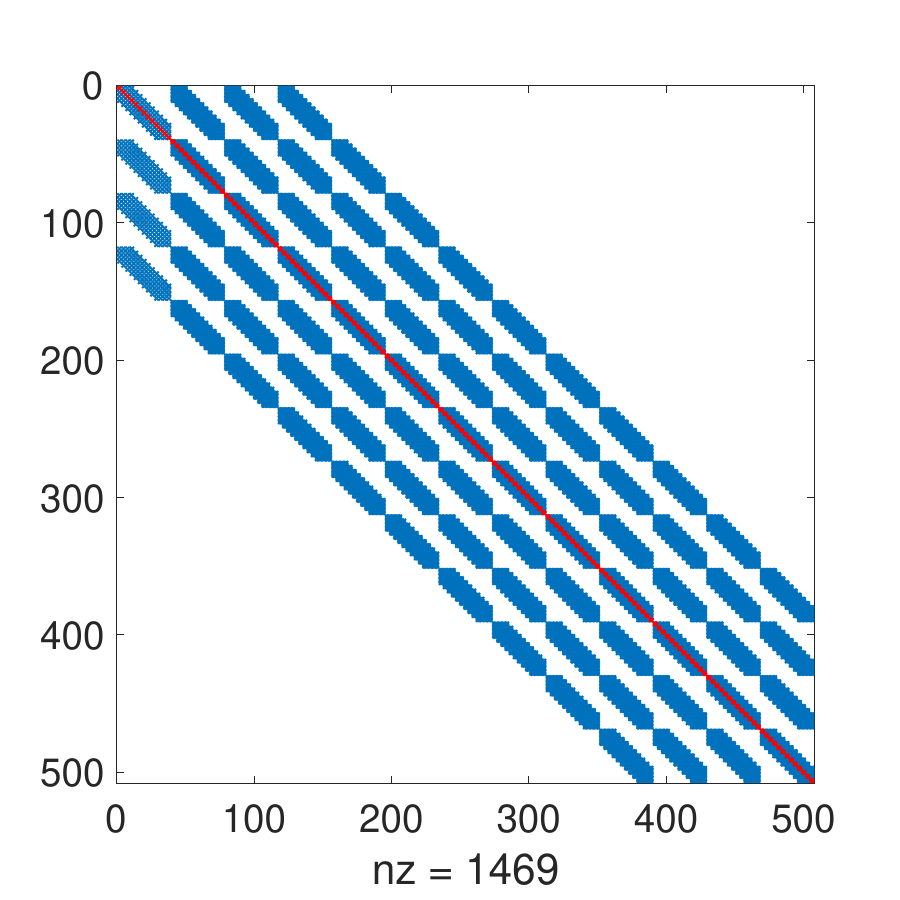}
	\includegraphics[width=0.45\textwidth,trim={0.6cm 1cm 1cm 1.2cm}, clip]{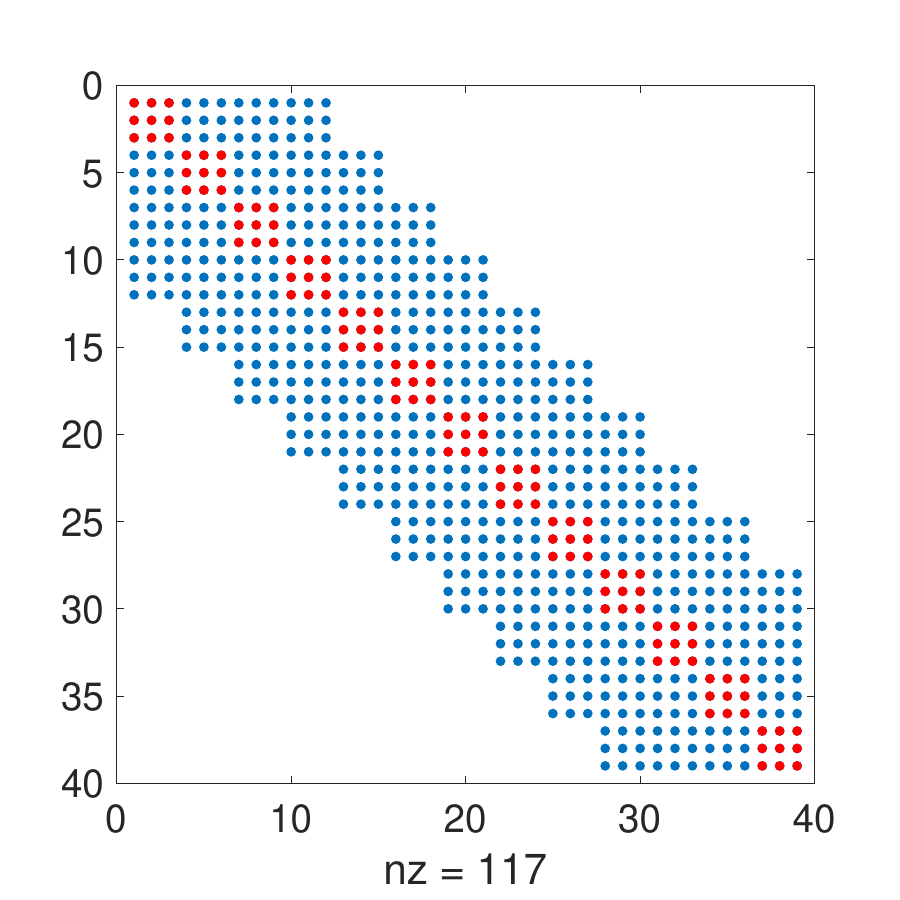}
	\caption{The left panel is the structure of the coefficient matrix of the linear system (\ref{continuous_traction}).  The right panel shows a close-up of one diagonal block.}\label{Mass_matrix}
\end{figure}

\subsection{Gaussian source}\label{gaussian_source}
In this section, we perform a numerical simulation with a Gaussian source at the top surface and verify that the curved mesh refinement interface does not generate any artifacts. 

We choose a flat top and bottom surface geometry 
\begin{equation*}
\theta_t\big(r^{(1)},r^{(2)}\big) = 1000,\quad \theta_b\big(r^{(1)},r^{(2)}\big) = 0,
\end{equation*}
respectively. The mesh refinement interface is parameterized by
\begin{equation}\label{interface_gausian}
\theta_i\big(r^{(1)},r^{(2)}\big) = 800+20\sin(4\pi r^{(1)})+20\cos(4\pi r^{(2)}),
\end{equation}
where $0\leq r^{(1)}, r^{(2)}, r^{(3)}\leq 1$. 
In addition, the mapping in the coarse domain $\Omega^c$ and fine domain $\Omega^f$ are given by 
\[ {\bf x} = {\bf X}^c({\bf r}) = \left(\begin{array}{c}
2000 r^{(1)}\\
2000 r^{(2)}\\
r^{(3)} \theta_i\big(r^{(1)}, r^{(2)}\big) + (1-r^{(3)}) \theta_b\big(r^{(1)},r^{(2)}\big) \end{array}\right) \]
and 
\[ {\bf x} = {\bf X}^f({\bf r}) = \left(\begin{array}{c}
2000 r^{(1)}\\
2000 r^{(2)}\\
r^{(3)}\theta_t\big(r^{(1)},r^{(2)}\big) + (1-r^{(3)})\theta_i\big(r^{(1)},r^{(2)}\big)\end{array}\right), \]
respectively. 
In the entire domain, we use the homogeneous material properties
\begin{equation*}
\rho(x^{(1)},x^{(2)},x^{(3)}) = 1.5\times 10^3,\ \  \mu(x^{(1)},x^{(2)},x^{(3)}) = 1.5\times 10^9,\ \ 
\lambda(x^{(1)},x^{(2)},x^{(3)})  = 3\times 10^9.
\end{equation*}

At the top surface, the Gaussian source 
${\bf g} = (g_1,g_2,g_3)^T$ is imposed as the Dirichlet data with $g_1 = g_2 = 0$ and 
\[g_3 = 10^9 \text{exp}\left(-\left(\frac{t-4/44.2}{1/44.2}\right)^2\right)\text{exp}\left(-\left(\frac{x^{(1)}-1000}{12.5}\right)^2-\left(\frac{x^{(2)}-1000}{12.5}\right)^2\right).\]  
Homogeneous Dirichlet boundary conditions are imposed at other boundaries. Both the initial conditions and the external forcing are set to zero everywhere. 
For these material properties, the shear wave velocity is $c_s = \sqrt{\mu/\rho}=1000$. With the dominant wave frequency $f_0=44.2\sqrt{2}/(2\pi)\approx 10$, the corresponding wavelength $c_s/f_0$ is approximately 100.

In the numerical schemes, we consider three different meshes: Mesh 1 is the Cartesian mesh without any interface and $n_1 = n_2 = 201, n_3 = 101$ with $n_i$ denotes the number of grid points in the direction $x^{(i)}$. This corresponds to 10 grid points per wavelength and is considered as the  reference solution. Mesh 2 is the curvilinear mesh with a curved mesh refinement interface defined in \eqref{interface_gausian} and $n_1^{2h} = n_2^{2h} = 101, n_3^{2h} = 41$, $n_1^h = n_2^h = 201, n_3^h = 21$. The mesh size in $\Omega^f$ is approximately the same as the mesh size in the Cartesian mesh. As a result, the waves are resolved with 5 grid points per wavelength in $\Omega^c$. Mesh 3 is obtained by refining Mesh 2 in all three spatial directions. 

\begin{figure}[htbp]
	\centering
	\includegraphics[width=0.49\textwidth,trim={0.05cm 0.1cm 0.55cm 0.45cm}, clip]{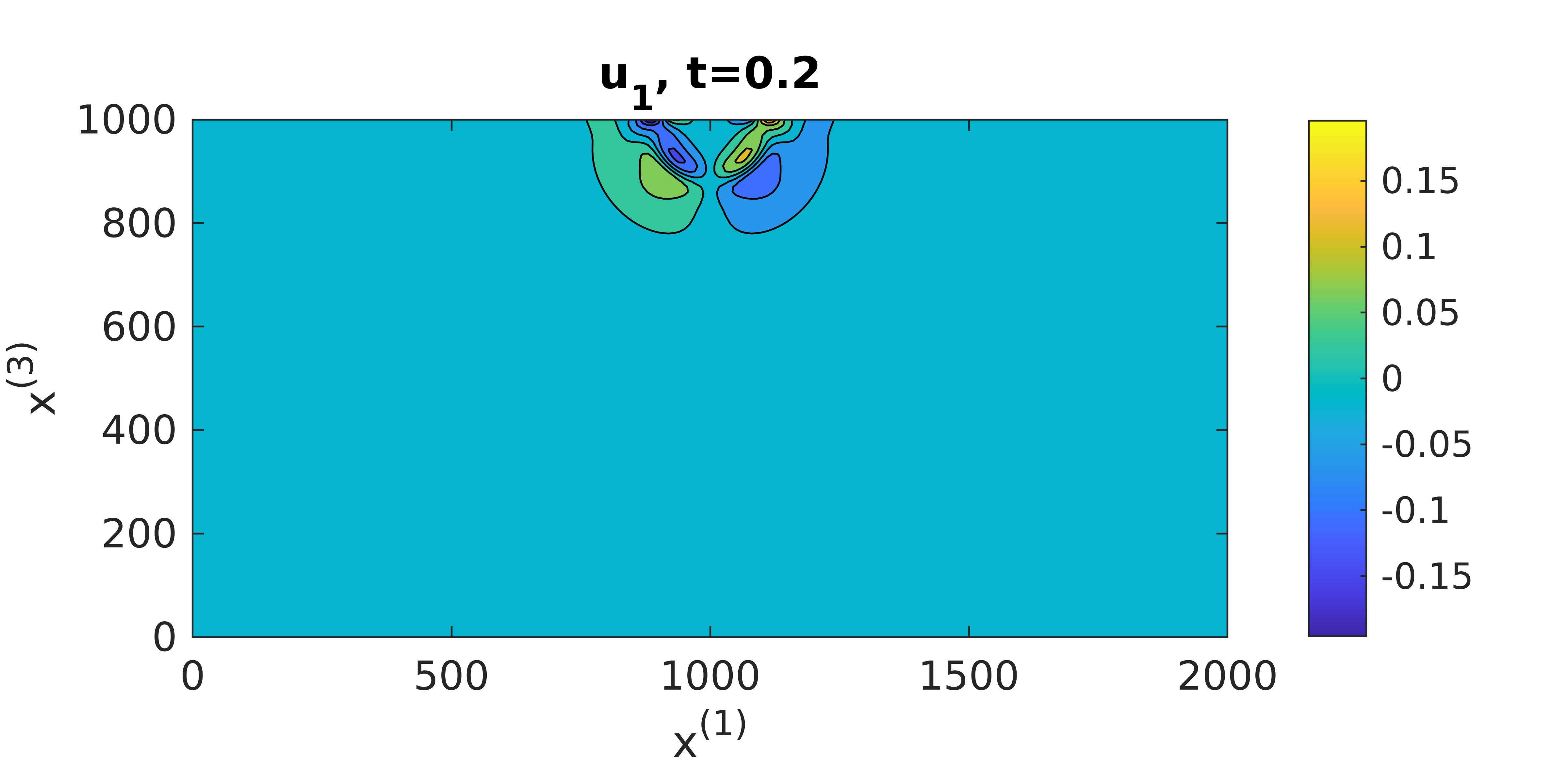}
	\includegraphics[width=0.49\textwidth,trim={0.05cm 0.1cm 0.55cm 0.45cm}, clip]{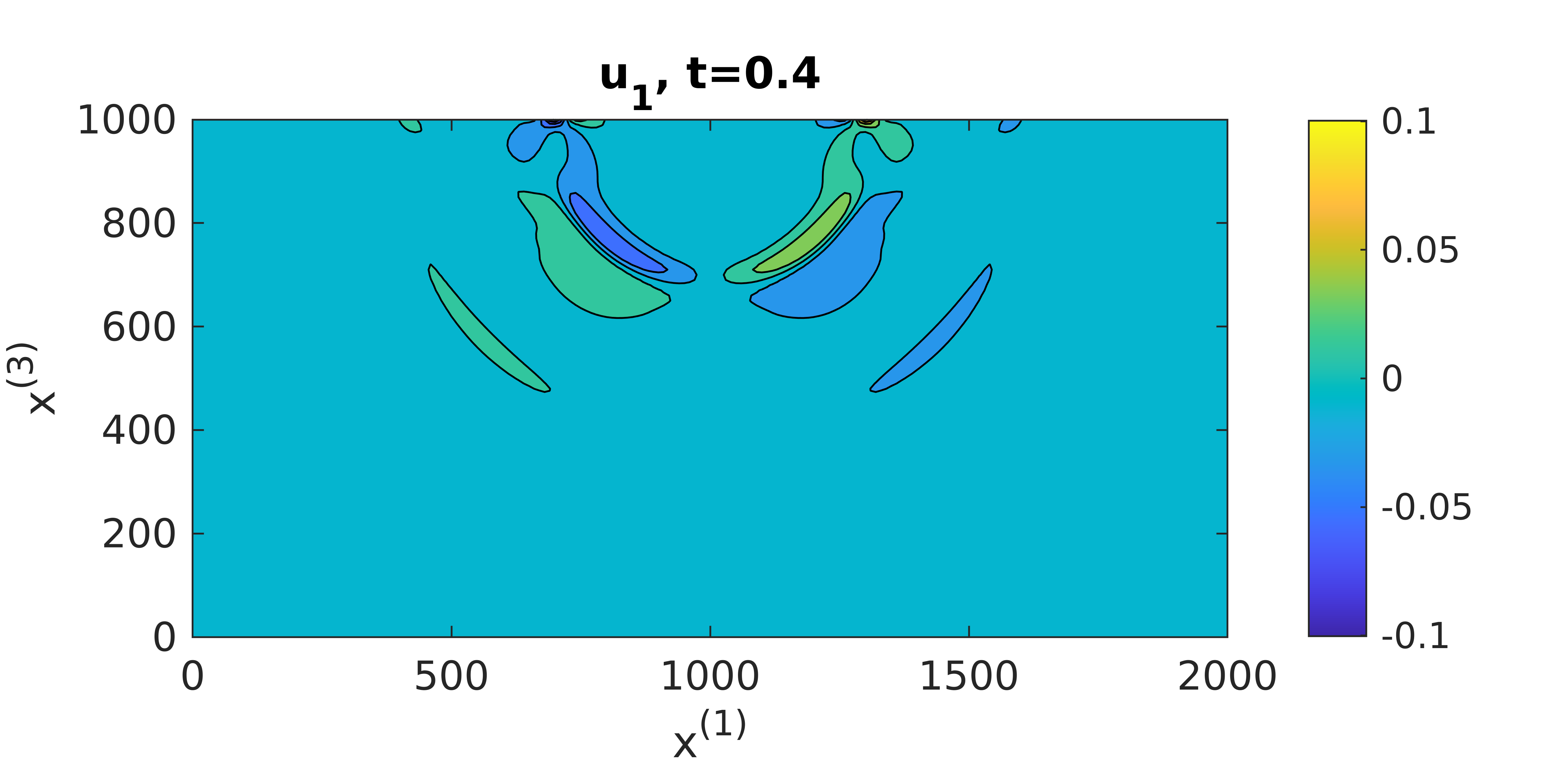}\\
	\includegraphics[width=0.49\textwidth,trim={0.05cm 0.1cm 0.55cm 0.45cm}, clip]{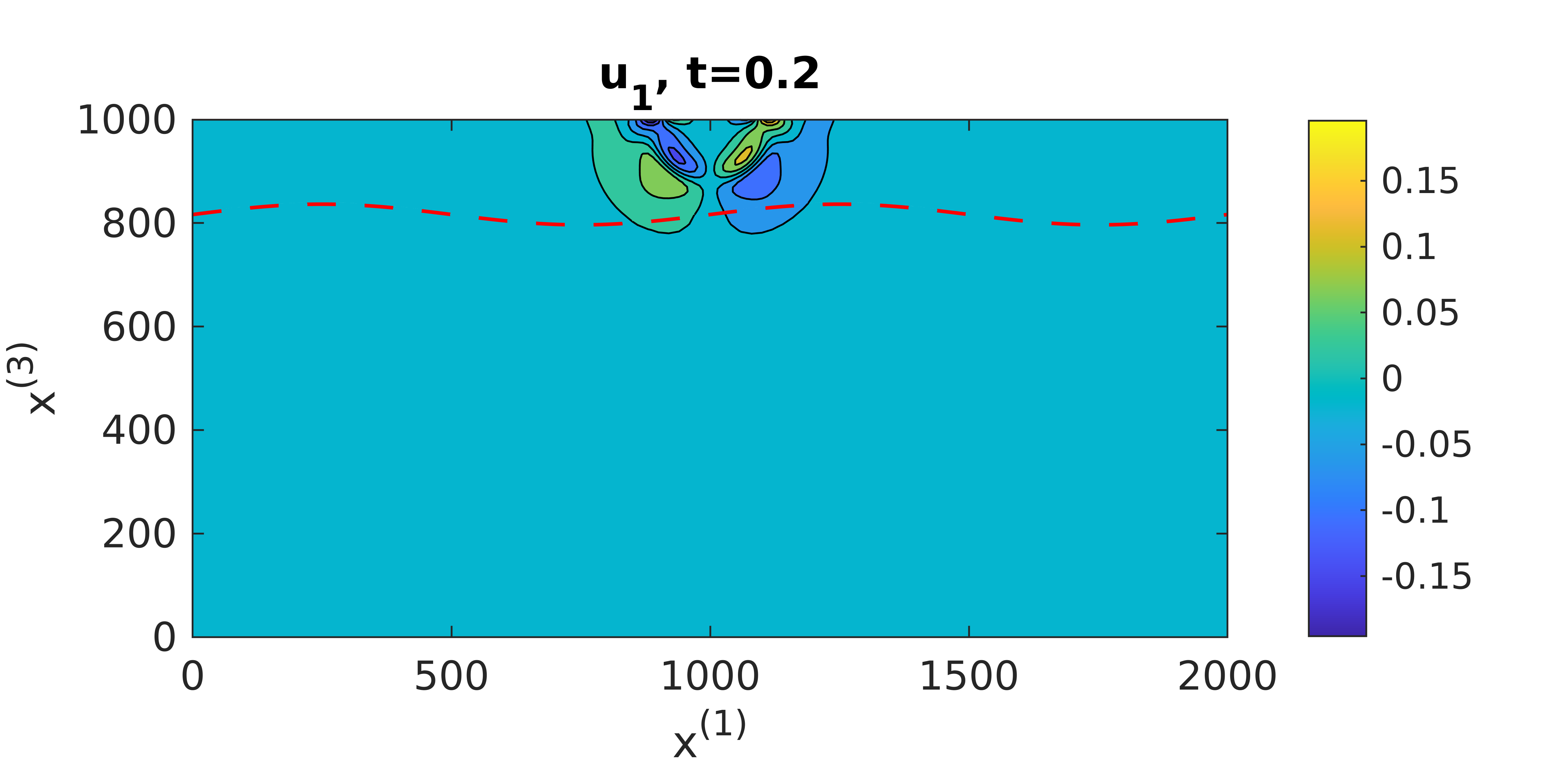}
	\includegraphics[width=0.49\textwidth,trim={0.05cm 0.1cm 0.55cm 0.45cm}, clip]{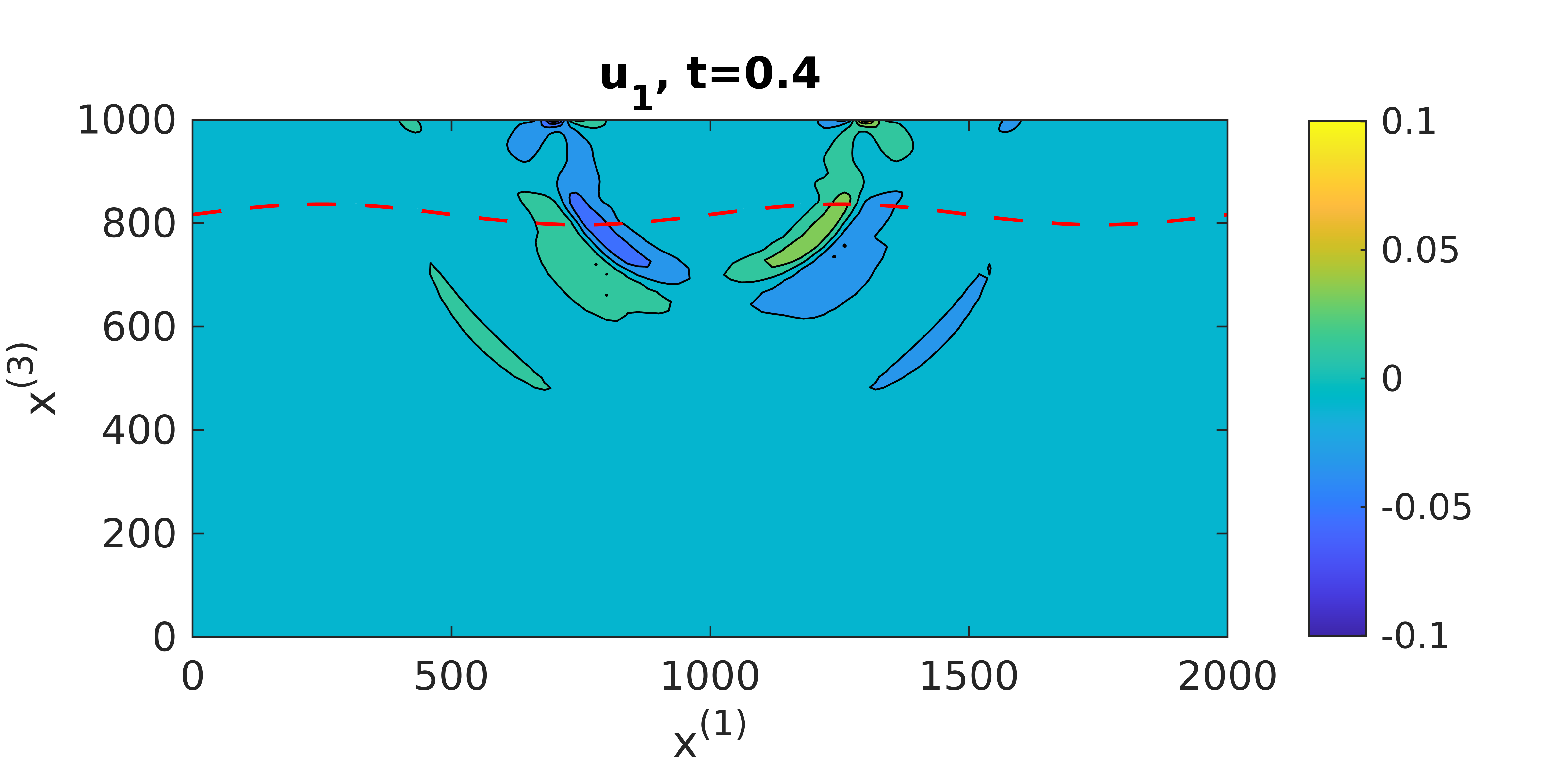}\\
	\includegraphics[width=0.49\textwidth,trim={0.05cm 0.1cm 0.55cm 0.45cm}, clip]{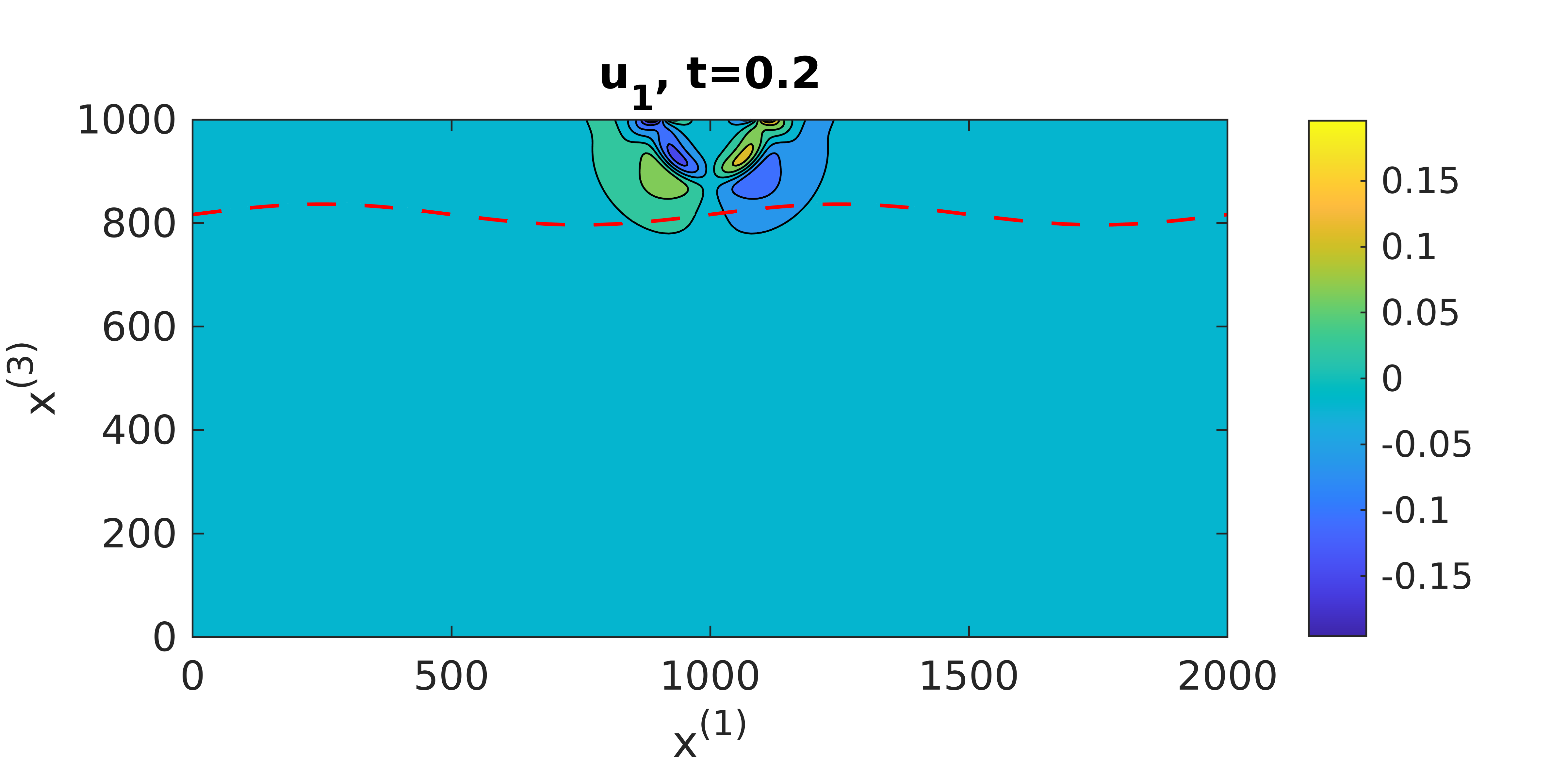}
	\includegraphics[width=0.49\textwidth,trim={0.05cm 0.1cm 0.55cm 0.45cm}, clip]{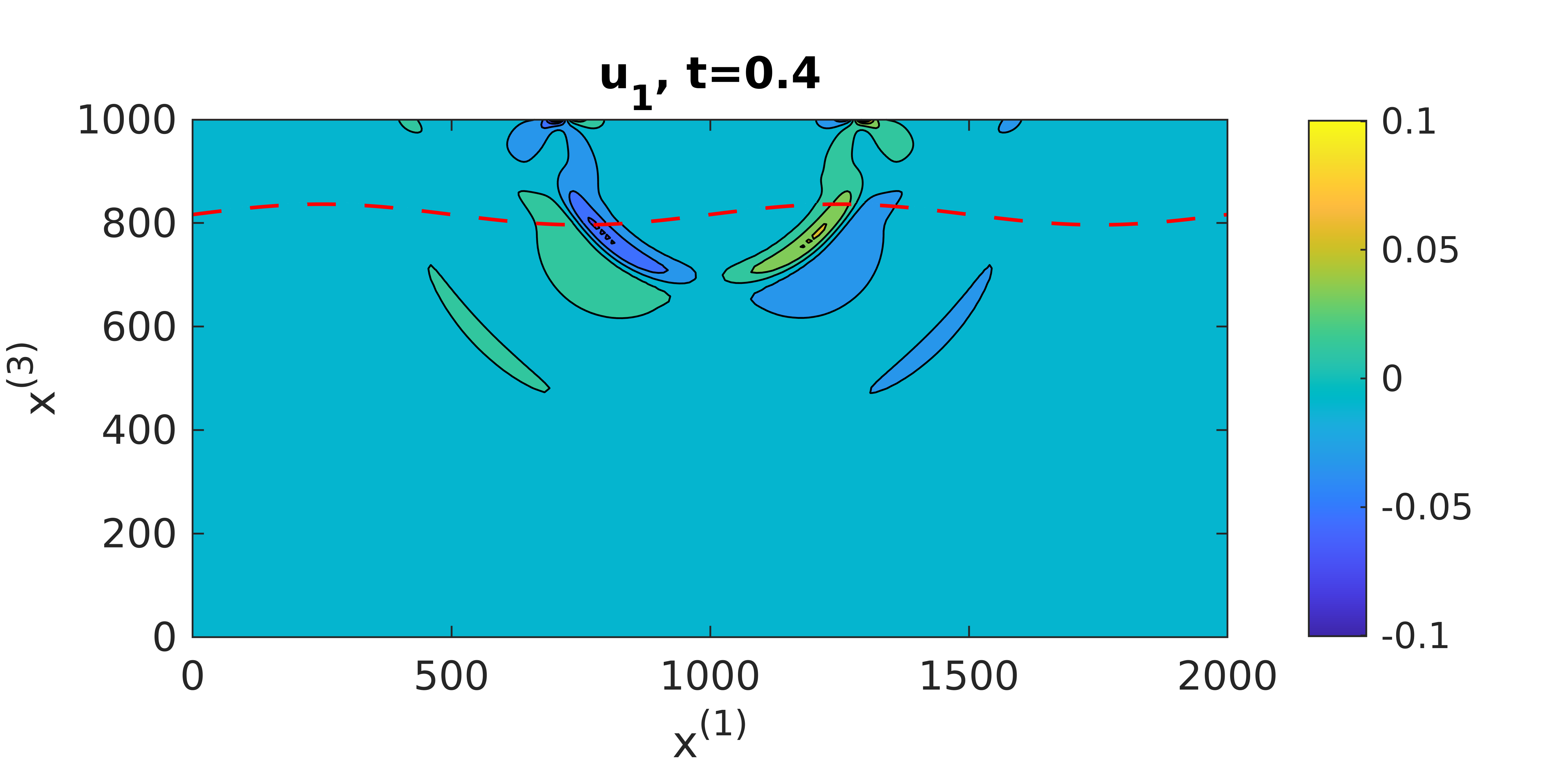}
\caption{The graphs for $u_1$. In the top, middle and bottom panel, we show numerical solutions at $t=0.2$ and $t=0.4$ computed with Mesh 1 (uniform Cartesian grid without any interface), Mesh 2 (curved interface) and Mesh 3 (a refinement of Mesh 2), respectively. The curved interfaces are marked with the red dash lines.}
\label{u1}
\end{figure}


\begin{figure}[htbp]
	\centering
	\includegraphics[width=0.49\textwidth,trim={0.05cm 0.1cm 0.55cm 0.45cm}, clip]{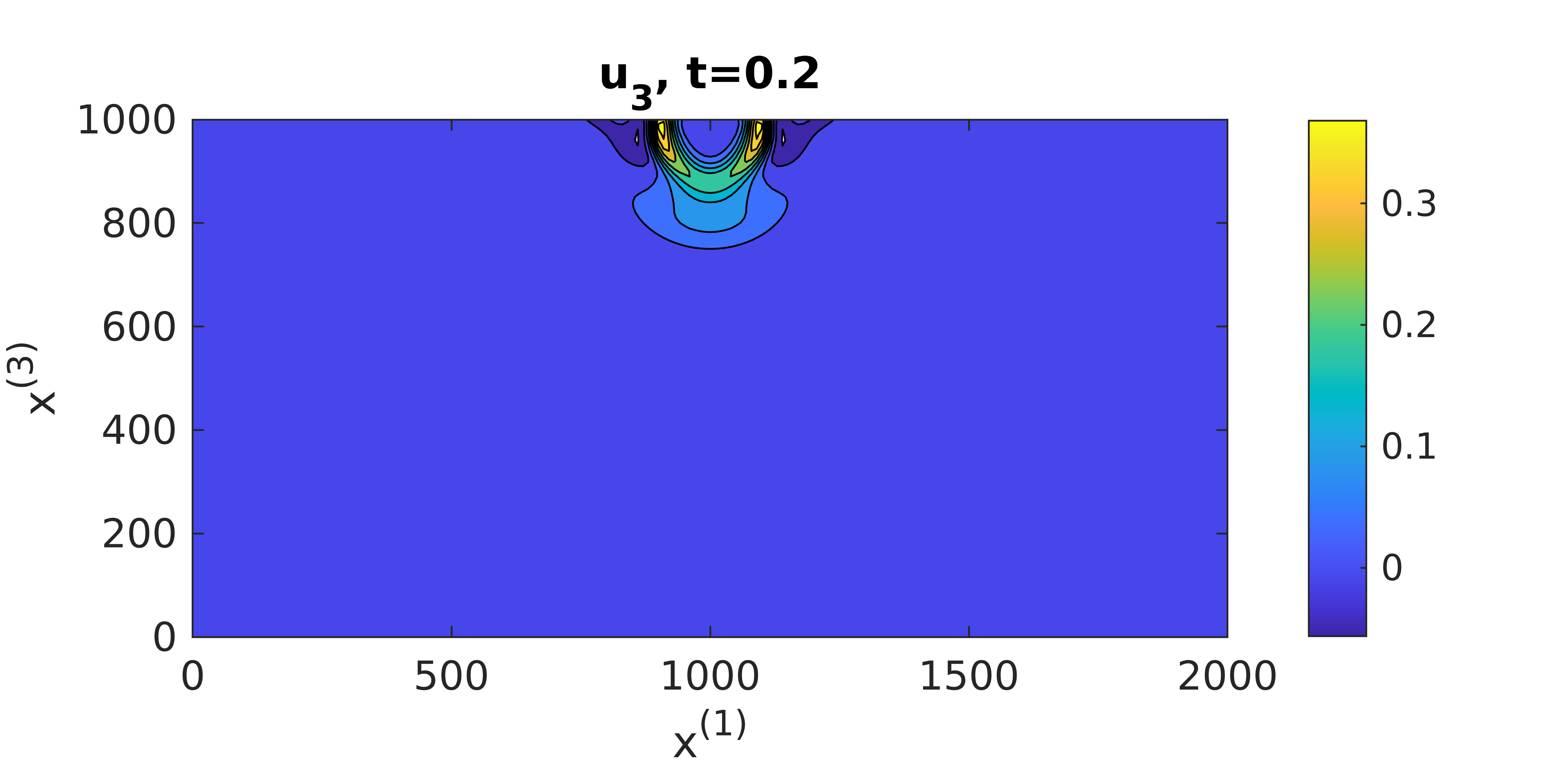}
	\includegraphics[width=0.49\textwidth,trim={0.05cm 0.1cm 0.55cm 0.45cm}, clip]{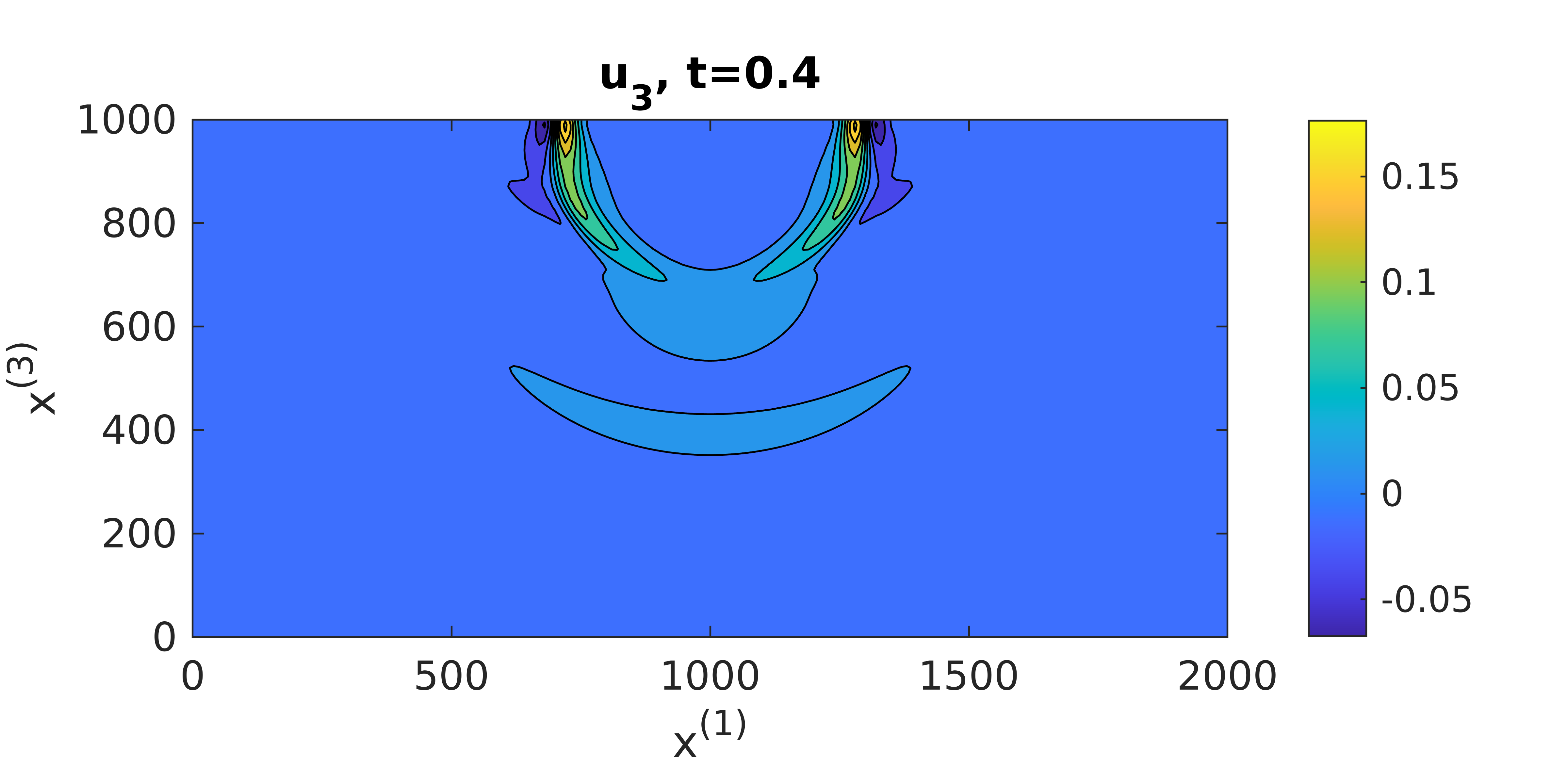}\\
	\includegraphics[width=0.49\textwidth,trim={0.05cm 0.1cm 0.55cm 0.45cm}, clip]{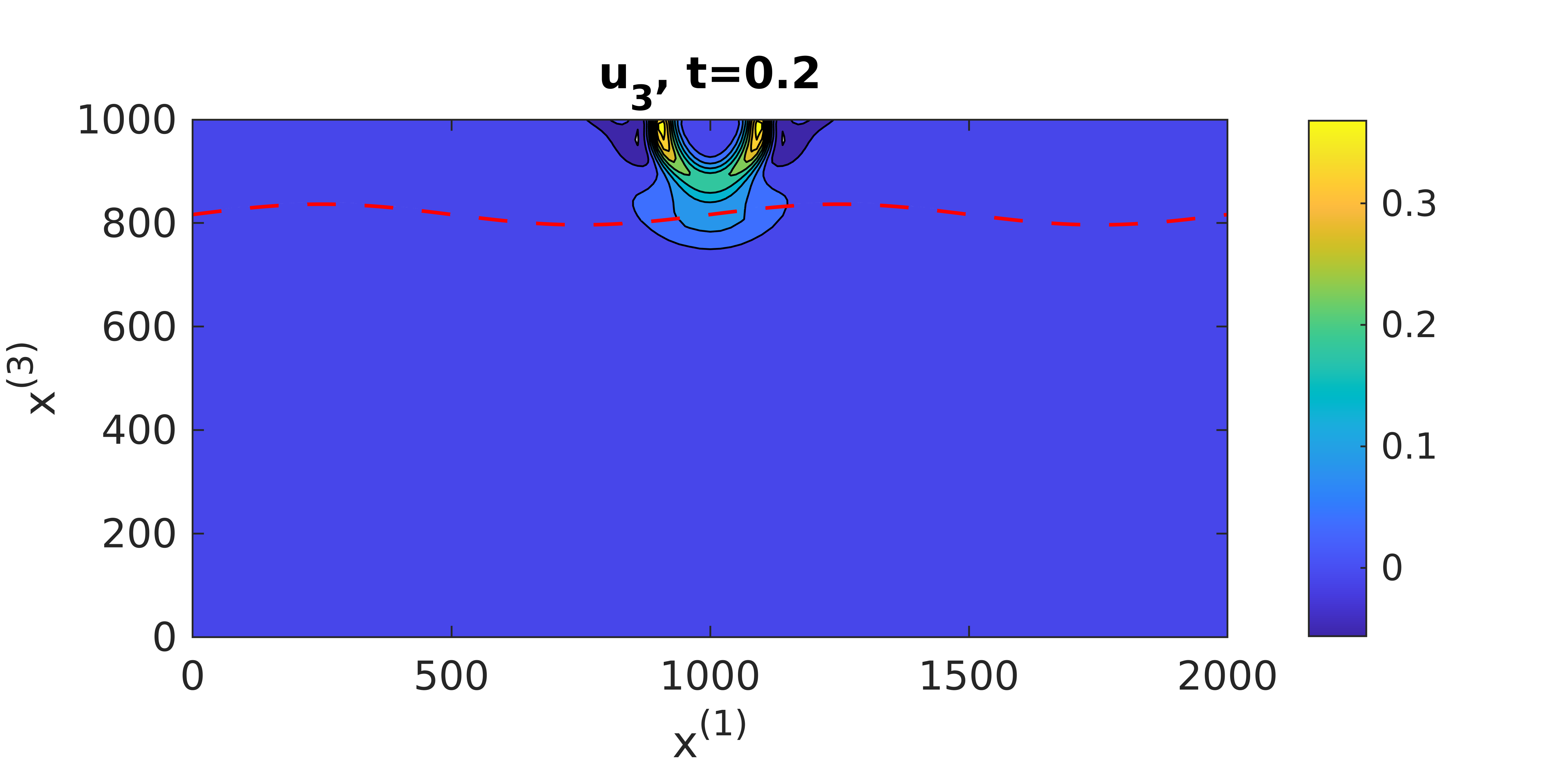}
	\includegraphics[width=0.49\textwidth,trim={0.05cm 0.1cm 0.55cm 0.45cm}, clip]{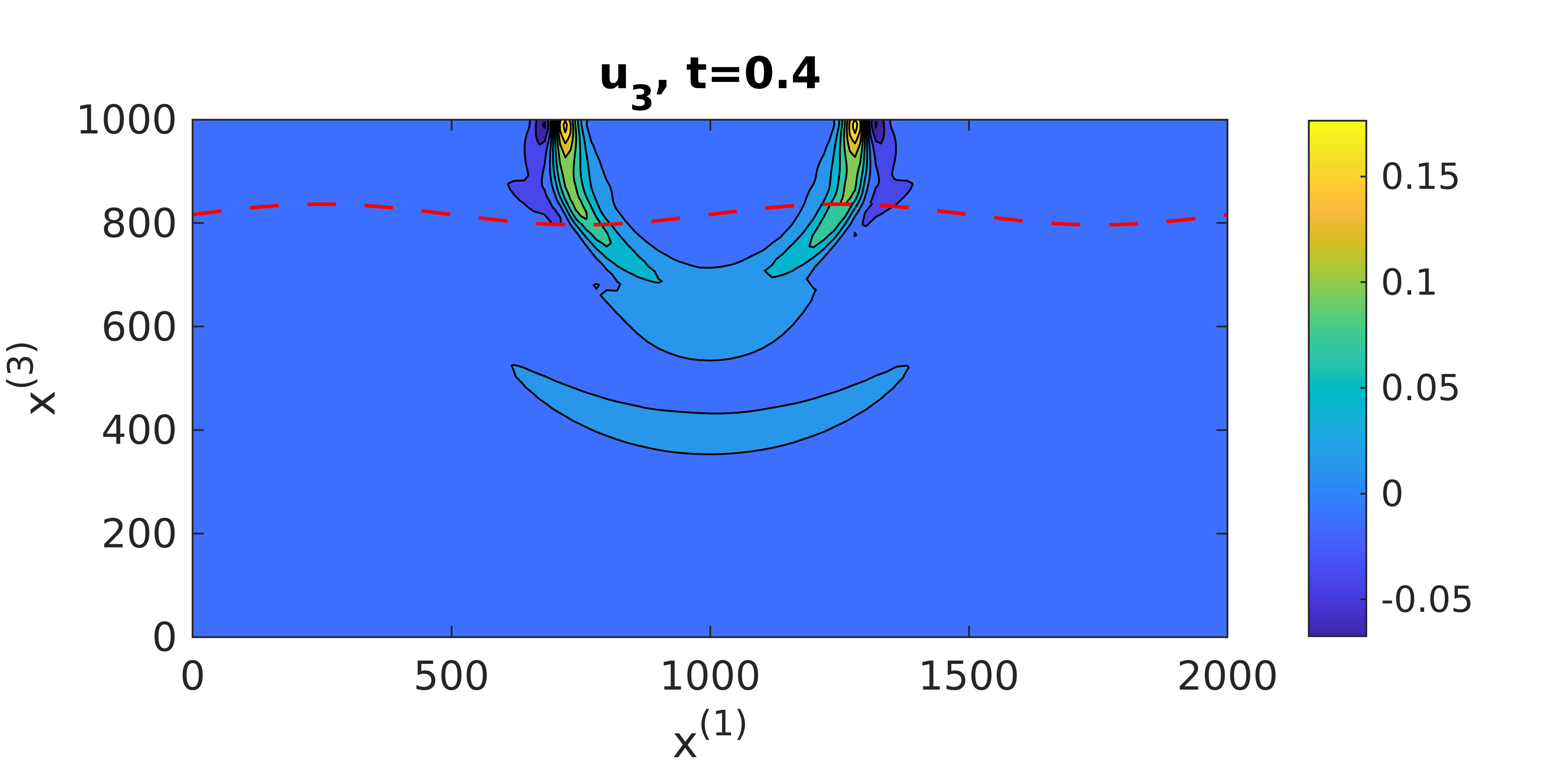}\\
	\includegraphics[width=0.49\textwidth,trim={0.05cm 0.1cm 0.55cm 0.45cm}, clip]{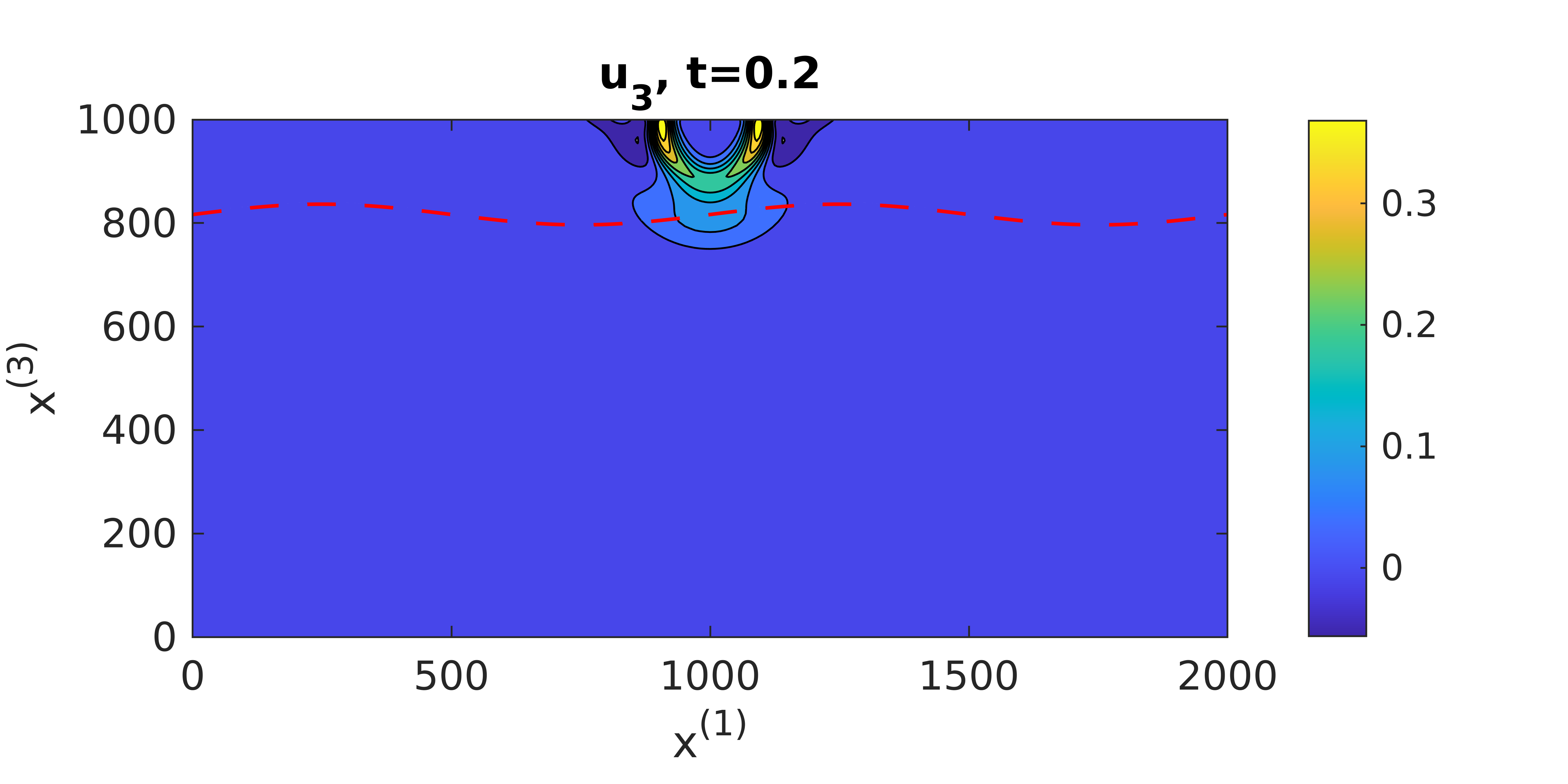}
	\includegraphics[width=0.49\textwidth,trim={0.05cm 0.1cm 0.55cm 0.45cm}, clip]{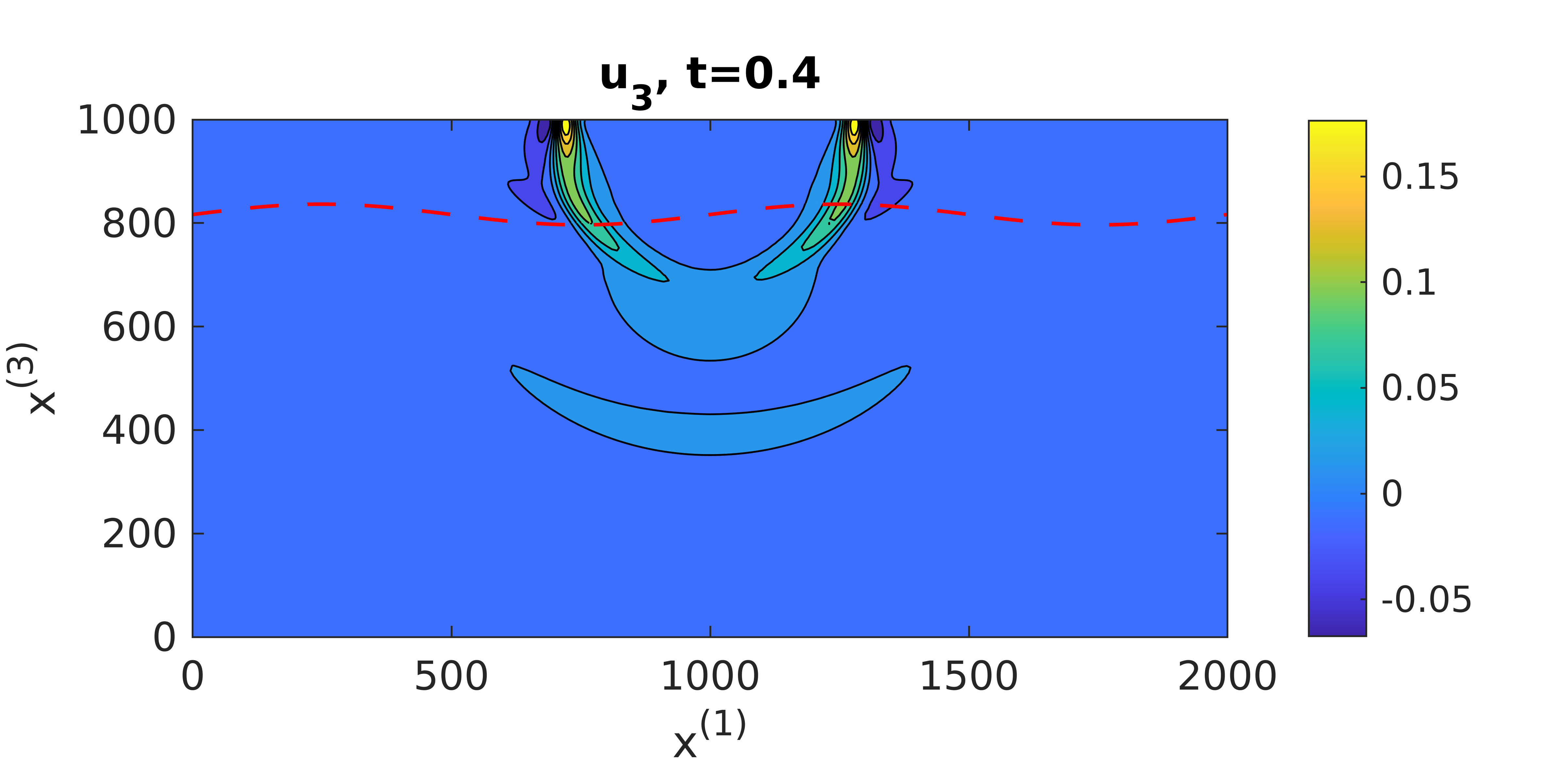}
	\caption{The graphs for $u_3$. In the top, middle and bottom panel, we show numerical solutions at $t=0.2$ and $t=0.4$ computed with Mesh 1 (uniform Cartesian grid without any interface), Mesh 2 (curved interface) and Mesh 3 (a refinement of Mesh 2), respectively. The curved interfaces are marked with the red dash lines.}
\label{u3}
\end{figure}
In Figure \ref{u1}, we plot the component $u_1$ at $t=0.2$ and $t=0.4$.  Some artifacts are observed in the solution computed with the second mesh, which is due to the small number of grid points per wavelength in $\Omega^c$. The results become better when the finer curvilinear mesh is used. From Figure \ref{u3}, we observe that there is no obvious reflection at the mesh refinement interface for the component $u_3$, and we have a better result when a finer curvilinear mesh is used. The component $u_2$ is zero up to round-off error for both the Cartesian mesh and curvilinear meshes and is not presented here.

\subsection{Energy conservation test}\label{conserved_energy}
To verify the energy conservation property of the scheme, we perform computation without external source term, but with a Gaussian initial data centered at the origin of the computational domain. The computational domain is chosen to be the same as in Sec.~\ref{convergence_study}. The material property is heterogeneous and discontinuous: for the fine domain $\Omega^f$, the density varies according to
\[\rho^f(x^{(1)},x^{(2)},x^{(3)}) = 3 + \sin(2x^{(1)}+0.3)\cos(x^{(2)}+0.3)\sin(2x^{(3)}-0.2),\] 
and material parameters satisfy
\[\mu^f(x^{(1)},x^{(2)},x^{(3)}) = 2 + \cos(3x^{(1)}+0.1)\sin(3x^{(2)}+0.1)\sin(x^{(3)})^2,\]
\[\lambda^f(x^{(1)},x^{(2)},x^{(3)}) = 15 + \cos(x^{(1)}+0.1)\sin(4x^{(2)}+0.1)\sin(3x^{(3)})^2;\]
for the coarse domain $\Omega^c$, the density varies according to
\[\rho^c(x^{(1)},x^{(2)},x^{(3)}) = 2 + \sin(x^{(1)}+0.3)\sin(x^{(2)}+0.3)\sin(2x^{(3)}-0.2),\] 
and material parameters satisfy
\[\mu^c(x^{(1)},x^{(2)},x^{(3)}) = 3 + \sin(3x^{(1)}+0.1)\sin(3x^{(2)}+0.1)\sin(x^{(3)}),\]
\[\lambda^c(x^{(1)},x^{(2)},x^{(3)}) = 21 + \cos(x^{(1)}+0.1)\cos(x^{(2)}+0.1)\sin(3x^{(3)})^2.\]
The initial Gaussian data is given by ${\bf C}(\cdot,0) = {\bf F}(\cdot,0) = {\bf u}(\cdot,0) = (u_1(\cdot,0),u_2(\cdot,0),u_3(\cdot,0))^T$ with
\begin{align*}
	u_1(\cdot,0) &= \mbox{exp}\left(-\frac{(x^{(1)}-\pi)^2}{0.1}\right)\mbox{exp}\left(-\frac{(x^{(2)}-\pi)^2}{0.1}\right)\mbox{exp}\left(-\frac{(x^{(3)}-\pi)^2}{0.1}\right),\\
	u_2(\cdot,0) &= \mbox{exp}\left(-\frac{(x^{(1)}-\pi)^2}{0.2}\right)\mbox{exp}\left(-\frac{(x^{(2)}-\pi)^2}{0.2}\right)\mbox{exp}\left(-\frac{(x^{(3)}-\pi)^2}{0.2}\right),\\
	u_3(\cdot,0) &= \mbox{exp}\left(-\frac{(x^{(1)}-\pi)^2}{0.1}\right)\mbox{exp}\left(-\frac{(x^{(2)}-\pi)^2}{0.2}\right)\mbox{exp}\left(-\frac{(x^{(3)}-\pi)^2}{0.2}\right).
\end{align*}
 The grid spacing in the parameter space for the coarse domain $\Omega^c$ is $2h_1 = 2h_2 = 2h_3 = \frac{\pi}{24}$ and for the fine domain $\Omega^f$ is $h_1 = h_2 = h_3 = \frac{\pi}{48}$, that is we have $25\times25\times13$ grid points in the coarse domain $\Omega^c$ and $49\times49\times25$ grid points in the fine domain $\Omega^f$. 

The semi-discrete energy is given by $({\bf f}_t,({\rho}^h\otimes{\bf I}){\bf f}_t)_h + \mathcal{S}_h({\bf f},{\bf f}) + ({\bf c}_t,({\rho}^{2h}\otimes{\bf I}){\bf c}_t)_{2h} + \mathcal{S}_{2h}({\bf c},{\bf c})$, see (\ref{semi_energy_1}). By using the same approach as for the isotropic elastic wave equation, see \cite{petersson2015wave,sjogreen2012fourth},  the expression for the fully discrete energy reads 
\begin{align*}
&E^{n+1/2} = \left|\left|(\rho^h\otimes {\bf I})^{\frac{1}{2}}\frac{{\bf f}^{n+1}-{\bf f}^n}{\Delta t}\right|\right|_h^2 \!+ S_h({\bf f}^{n+1},{\bf f}^n) - \frac{(\Delta t)^2}{12}\Big((J^h\otimes{\bf I})^{-1}\mathcal{L}^h{\bf f}^{n+1},(\rho^h\otimes{\bf I})^{-1}(J^h\otimes{\bf I})^{-1}\mathcal{L}^h{\bf f}^n\Big)_h\\
&+ \left|\left|(\rho^{2h}\otimes{\bf I})^{\frac{1}{2}}\frac{{\bf c}^{n+1}-{\bf c}^n}{\Delta t}\right|\right|_{2h}^2 \!\!\!+ S_{2h}({\bf c}^{n+1},{\bf c}^n) \!-\! \frac{(\Delta t)^2}{12}\Big(\!(J^{2h}\otimes{\bf I})^{-1}\wt{\mathcal{L}}^{2h}{\bf c}^{n+1}\!\!,(\rho^{2h}\otimes{\bf I})^{-1}(J^{2h}\otimes{\bf I})^{-1}\wt{\mathcal{L}}^{2h}{\bf c}^n\!\Big)_{2h}.
\end{align*}
We plot the relative change in the fully discrete energy, $(E^{n+1/2}-E^{1/2})/E^{1/2}$, as a function of time with $t\in[0,120]$ in Figure \ref{discrete_energy}. This corresponds to $6186$ time steps. Clearly, the fully discrete energy remains constant up to the round-off error.
\begin{figure}[htbp]
	\centering
	\includegraphics[width=0.6\textwidth,trim={0cm 0cm 0cm 0cm}, clip]{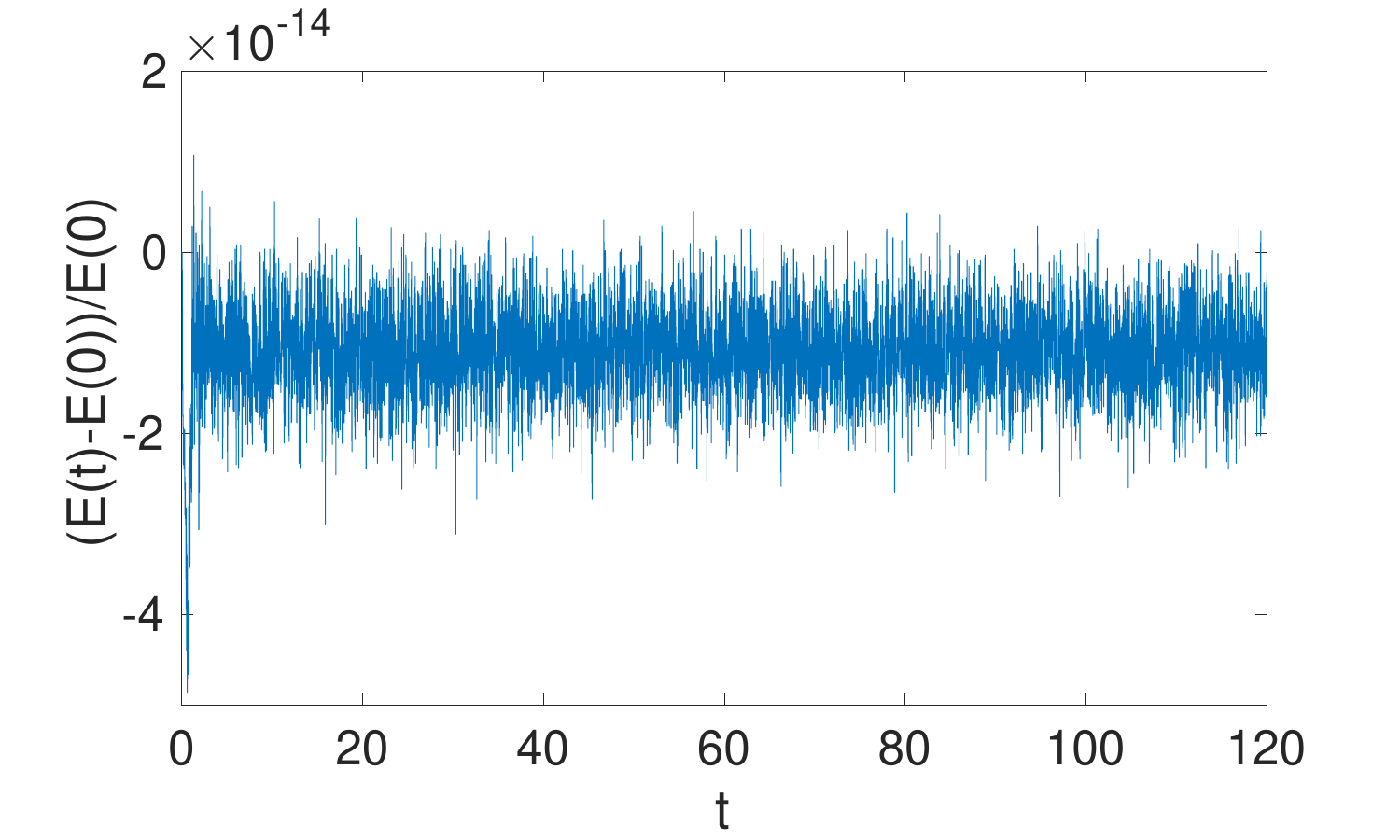}
	\caption{The relative change in the fully discrete energy as a function of time. Here, $t = 120$ corresponds to $6186$ time steps.}\label{discrete_energy}
\end{figure}

\subsection{LOH.1 model problem with layered material}
As the final numerical example, we consider the layer-over-halfspace benchmark problem LOH.1 \cite{Day2001}. The computational domain is taken to be $(x,y,z)\in[0,30000]^2\times[0,17000]$ with a free surface boundary conditions at $z=0$.   The problem is driven by a single point moment source defined as 
$g(t,t_0,\omega) \mathcal{M} \cdot \nabla\delta (\mathbf{x}-\mathbf{x_0})$, 
where the point source location is $\mathbf{x_0}= (15000, 15000, 2000)$  and the moment time function is
\[g(t,t_0,\omega) = \frac{\omega}{\sqrt{2\pi}}e^{-\omega^2(t - t_0)^2/2}, \ \ \ \omega = 16.6667,\ \ \ \ t_0 = 0.36.\]
In the 3-by-3 symmetric moment tensor $\mathcal{M}$, the only nonzero elements are $\mathcal{M}_{12}=\mathcal{M}_{21}=10^{18}$. The center frequency is ${\omega}/{(2\pi)}=2.65$ and the highest significant frequency is estimated to be $2.5{\omega}/{(2\pi)}=6.63$.

The LOH.1 model has a layered material property with a material discontinuity at $z=1000$, with the dynamic and mechanical parameters given in Table \ref{material_parameter}. In the top layer $z\in [0, 1000]$, both the compressional and shear velocity are lower than the rest of the domain. For computational efficiency, a smaller grid spacing shall be used in the top layer. 

\begin{table}[htbp]
	\begin{center}
		\begin{tabular}{c c c c c}
			\hline
			~   & Depth $[m]$& $V_p[m/s]$ & $V_s [m/s]$ & $\rho[Kg/m^3]$ \\
			\hline
			Layer&0--1000& 4000& 2000& 2600\\
			half-space &1000--17000 & 6000 & 3464& 2700\\
			\hline 
		\end{tabular}
	\end{center}
	\caption{Dynamic and mechanical parameters for the layer and the lower half-space of the layer over half-space test.}\label{material_parameter}
\end{table} 

We solve the LOH.1 model problem by using the open source code SW4, where our proposed method has been implemented. The solution is recorded in a receiver on the free surface at the point $(x, y, z) = (21000, 23000, 0)$. The time history of the vertical, transverse and radial velocities are shown in Figure \ref{loh1_100} with grid spacing $h = 100$ in the half-space and $h/2 = 50$ in the top layer. With the highest significant frequency 6.63 Hz, the smallest number of grid points per wavelength is only 5.22. Despite this, we observe  the numerical solutions agree well with the exact solution. In Figure \ref{loh1_50}, the solutions computed on a finer mesh with $h = 50$ in the half-space and $h/2 = 25$ in the top layer look identical to the exact solutions.
\begin{figure}[htbp]
	\centering
	\includegraphics[width=0.9\textwidth,trim={4cm 0.2cm 4cm 1.5cm}, clip]{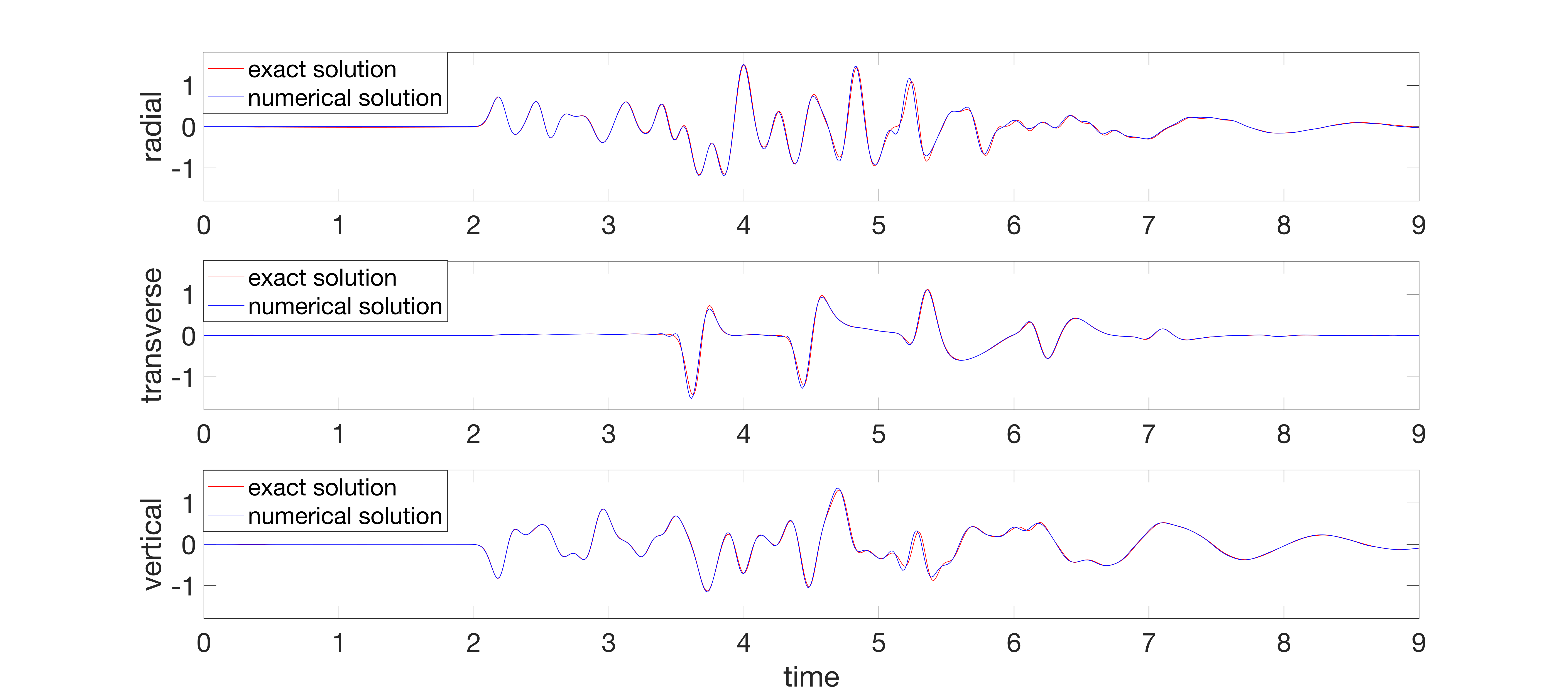}
	\caption{LOH.1: The radial (top), transverse (middle), and vertical (bottom) velocities time histories. Here the numerical solutions are plotted in blue ($h = 100$) and the semi-analytical solution is plotted in red.}\label{loh1_100}
\end{figure}

\begin{figure}[htbp]
	\centering
	\includegraphics[width=0.9\textwidth,trim={4cm 0.2cm 4cm 1.5cm}, clip]{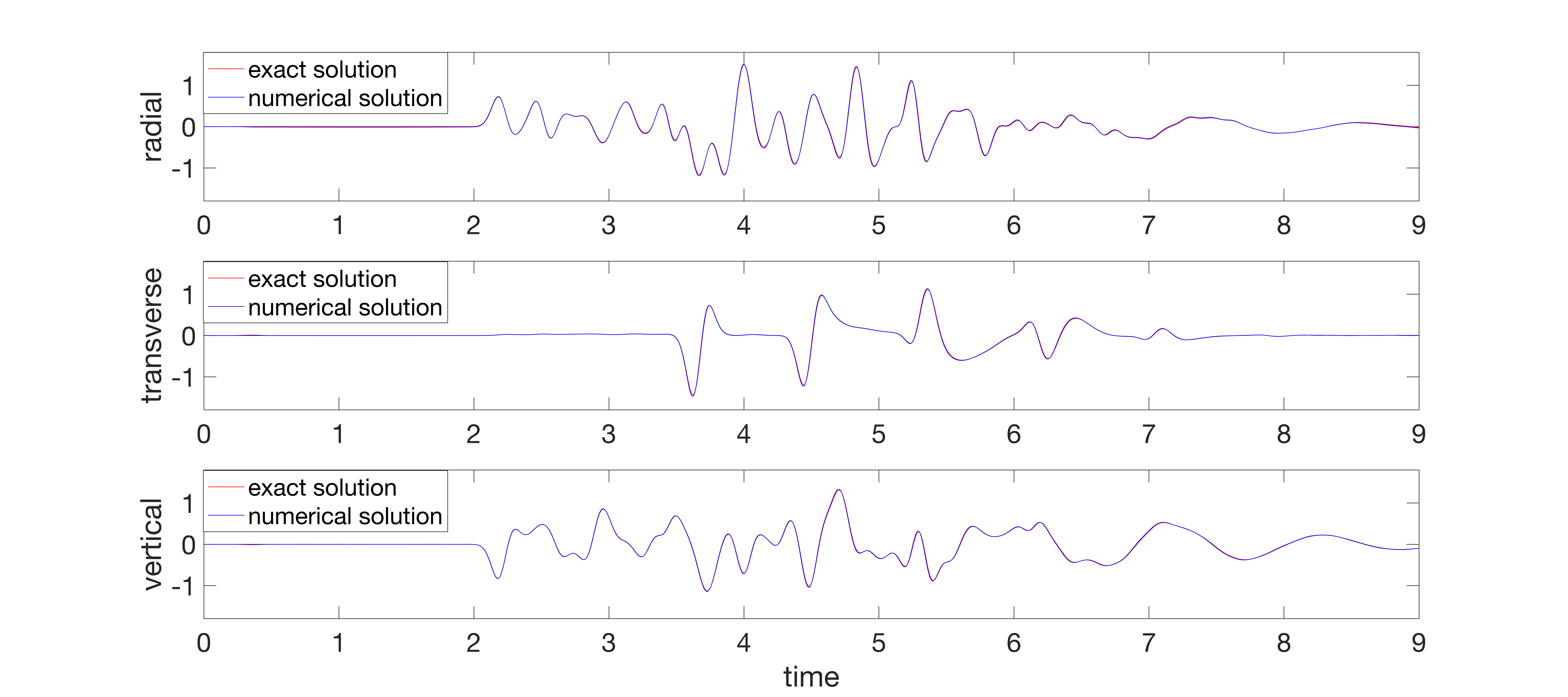}
	\caption{LOH.1: The radial (top), transverse (middle), and vertical (bottom) velocities time histories. Here the numerical solutions are plotted in blue ($h = 50$) and the semi-analytical solution is plotted in red.}\label{loh1_50}
\end{figure}

To test the performance of the new method, we record the quotient between the computational time of solving the linear system for the mesh refinement interface and of the time-stepping procedure in Table \ref{time}. We have run simulations on two different computer clusters. First, we use two nodes on the Rackham cluster with each node consisting of two 10-core Intel Xeon V4 CPUs and 128 GB memory. In the second simulation, we use three nodes on ManeFrame II (M2) with each node consisting of two 18-core Intel Xeon E5-2695 v4 CPUs and 256 GB memory. From Table \ref{time}, we observe that our new method (with ghost points from the coarse domain) needs much less time to solve the linear system for interface conditions compared with the old method in SW4 (with ghost points from both coarse and fine domains). 

\begin{table}[htbp]
	\begin{center}
		\begin{tabular}{|c|c|c|}
			\hline
			Machine   & new method & old method \\
			\hline
			Rackham & 4.02\% &  8.16\%\\
			\hline
			M2 &5.17\% & 8.87\%\\
			\hline 
		\end{tabular}
	\end{center}
		\caption{The quotient of the computational time of solving the linear system for the mesh refinement interface and of the time-stepping procedure.}\label{time}
\end{table} 

In addition, the proposed method implemented in SW4 has excellent parallel scalability. When running the same model problem with 4 nodes (80 cores) on the Rackham cluster, the computational time of the time stepping procedure is $51\%$ of that with 2 nodes. Further increasing to 8 nodes (160 cores), the computational time of the time stepping procedure is $52\%$ of that with 4 nodes.

\section{Conclusion}
We have developed a fourth order accurate finite difference method for the three dimensional elastic wave equations in heterogeneous media. To take into account discontinuous material properties, we partition the domain into subdomains such that interfaces are aligned with material discontinuities such that the material property is smooth in each subdomain. Adjacent subdomains are coupled through physical interface conditions: continuity of displacements and continuity of traction.

In a realistic setting, these subdomains have curved faces. We use a coordinate transformation and discretize the governing equations on curvilinear meshes. In addition, we allow nonconforming mesh refinement interfaces such that the mesh sizes in each block need not to be the same. With this important feature, we can choose the mesh sizes according to the velocity structure of the material and keep the grid points per wavelength almost the same in the entire domain. 

The finite difference discretizations satisfy a summation-by-parts property. At the interfaces, physical interface conditions are imposed by using ghost points and mesh refinement interfaces with hanging nodes are treated numerically by the fourth order interpolation operators. Together with a fourth order accurate predictor-corrector time stepping method, the fully discrete equation is energy conserving. We have conducted numerical experiments to verify the energy conserving property, and the fourth order convergence rate. Furthermore, our numerical experiments indicate that there is little artificial reflection at the interface.

To obtain values of the ghost points, a system of linear equations must be solved. In our formulation, we only use ghost points from the coarse domain, which is more efficient than the traditional approach of using ghost points from both domains.  For large-scale simulations in three dimensions, the LU factorization cannot be used due to memory limitations. We have studied and compared three iterative methods for solving the linear system.

Our proposed method has been implemented in the open source code SW4 \cite{SW4}, which can be used to solve realistic seismic wave propagation problems on large parallel, distributed memory, machines. We have tested the benchmark problem LOH.1 and verified the improved efficiency. 

\section*{Acknowledgement}

This work was performed under the auspices of the U.S. Department of Energy by Lawrence Livermore National Laboratory under Contract DE-AC52-07NA27344. This is contribution LLNL-JRNL-810518. This research was supported by the Exascale Computing Project (ECP), project 17-SC-20-SC, a collaborative effort of two U.S. Department of Energy (DOE) organizations - the Office of Science and the National Nuclear Security Administration.

Part of the computations were performed on resources provided by Swedish National Infrastructure for Computing (SNIC) through Uppsala Multidisciplinary Center for Advanced Computational Science (UPPMAX) under Project SNIC 2019/8-263.

\appendix

\section{Terms in the spatial discretization}\label{appendix_cdomain}
For the first term in (\ref{L_operator}), we have
\begin{align*}
{Q}_l^{2h}({N}_{ll}^{2h}){\bf c} := \left(\begin{array}{c}
({Q}_l^{2h}({N}_{ll}^{2h}){\bf c})_1 \\
({Q}_l^{2h}({N}_{ll}^{2h}){\bf c})_2 \\
({Q}_l^{2h}({N}_{ll}^{2h}){\bf c})_3 
\end{array}\right), \quad ({Q}_l^{2h}({N}_{ll}^{2h}){\bf c})_p = \sum_{q = 1}^{3} Q_l^{2h}(N_{ll}^{2h}(p,q)) {c}^{(q)},\quad p = 1,2,3,
\end{align*}
where we have used a matlab notation $N_{ll}^{2h}(p,q)$ to represent the $p$-th row and $q$-th column of the matrix $N_{ll}^{2h}$; $Q_l^{2h}(N_{ll}^{2h}(p,q)){ c}^{(q)}$ is the central difference operator in direction $r^{(l)}$ for spatial second derivative with variable coefficient. For the second term in (\ref{L_operator}), we have
\begin{align*}
\wt{{G}}_3^{2h}({N}_{33}^{2h}){\bf c} := \left(\begin{array}{c}
(\wt{{G}}_3^{2h}({N}_{33}^{2h}){\bf c})_1 \\
(\wt{{G}}_3^{2h}({N}_{33}^{2h}){\bf c})_2 \\
(\wt{{G}}_3^{2h}({N}_{33}^{2h}){\bf c})_3 
\end{array}\right), \quad (\wt{{G}}_3^{2h}({N}_{33}^{2h}){\bf c})_p = \sum_{q = 1}^{3} \wt{G}_3^{2h}(N_{33}^{2h}(p,q)) {c}^{(q)},\quad p = 1,2,3,
\end{align*}
where $\wt{G}_3^{2h}(N_{33}^{2h}(p,q)) {c}^{(j)}$ is the second derivative SBP operator  defined in (\ref{sbp_2nd_1}) for direction $r^{(3)}$. For the third term in (\ref{L_operator}), we have
\begin{align*}
{D}_l^{2h}({N}_{lm}^{2h}{D}_m^{2h}{\bf c}) := \left(\begin{array}{c}
(D^{2h}_l(N^{2h}_{lm}D_m^{2h}{\bf c}))_1 \\
(D^{2h}_l(N^{2h}_{lm}D_m^{2h}{\bf c}))_2 \\
(D^{2h}_l(N^{2h}_{lm}D_m^{2h}{\bf c}))_3 
\end{array}\right), (D^{2h}_l(N^{2h}_{lm}D_m^{2h}{\bf c}))_p = \sum_{q = 1}^{3} D^{2h}_l(N^{2h}_{lm}(p,q)D_m^{2h}{c}^{(q)}), p = 1,2,3.
\end{align*}
Here, $D_m^{2h}c^{(q)}$ is a central difference operator in direction $r^{(m)}$ for the spatial first derivative, and $D_3^{2h}c^{(q)}$ is the SBP operator defined in (\ref{first_sbp}) for direction $r^{(3)}$.

For the second term in (\ref{Lf_operator}), we have
\begin{align*}
{{G}}_3^{h}({N}_{33}^{h}){\bf f} := \left(\begin{array}{c}
({{G}}_3^{h}({N}_{33}^{h}){\bf f})_1 \\
({{G}}_3^{h}({N}_{33}^{h}){\bf f})_2 \\
({{G}}_3^{h}({N}_{33}^{h}){\bf f})_3 
\end{array}\right), \quad ({{G}}_3^{h}({N}_{33}^{h}){\bf f})_p = \sum_{q = 1}^{3} {G}_3^{h}(N_{33}^{h}(p,q)) {f}^{(q)},\quad p = 1,2,3.
\end{align*}
Here, ${G}_3^{h}(N_{33}^{h}(p,q)) {f}^{(q)}$ is the SBP operator defined in (\ref{sbp_2nd_2}) for direction $r^{(3)}$. 

For the continuity of traction (\ref{continuous_traction}), we have \[\wt{\mathcal{A}}^{2h}_3{\bf c} = N_{31}^{2h}D_1^{2h}{\bf c}+N_{32}^{2h}D_2^{2h}{\bf c}+N_{33}^{2h}\wt{\mathcal D}_3^{2h}{\bf c},\]
where
\begin{align*}
{N}_{3l}^{2h}D_l^{2h}{\bf c} := \left(\begin{array}{c}
({N}_{3l}^{2h}D_l^{2h}{\bf c})_1 \\
({N}_{3l}^{2h}D_l^{2h}{\bf c})_2 \\
({N}_{3l}^{2h}D_l^{2h}{\bf c})_3 
\end{array}\right), \quad ({N}_{3l}^{2h}D_l^{2h}{\bf c})_p = \sum_{q = 1}^{3} N_{3l}^{2h}(p,q)D_l^{2h}{c}^{(q)},\quad l = 1,2, \quad p = 1,2,3
\end{align*}
with $D_l^{2h}{c}^{(q)}$ to be a central difference operator for first spatial derivative in direction $r^{(l)}$, and
\begin{align*}
{N}_{33}^{2h}\wt{\mathcal{D}}_3^{2h}{\bf c} := \left(\begin{array}{c}
({N}_{33}^{2h}\wt{\mathcal{D}}_3^{2h}{\bf c})_1 \\
({N}_{33}^{2h}\wt{\mathcal{D}}_3^{2h}{\bf c})_2 \\
({N}_{33}^{2h}\wt{\mathcal{D}}_3^{2h}{\bf c})_3 
\end{array}\right), \quad ({N}_{33}^{2h}\wt{\mathcal{D}}_3^{2h}{\bf c})_p = \sum_{q = 1}^{3} N_{33}^{2h}(p,q)\wt{\mathcal{D}}_3^{2h}{c}^{(q)},\quad p = 1,2,3
\end{align*}
with $\wt{\mathcal{D}}_3^{2h}{c}^{(q)}$ to be the difference operator for first spatial derivative in direction $r^{(3)}$ defined as in the second equation of (\ref{sbp_1st_1}); and 
\[\mathcal{A}_3^h{\bf f} = N_{31}^hD_1^h{\bf f}+N_{32}^hD_2^h{\bf f}+N_{33}^h\mathcal{D}_3^h{\bf f},\]
where
\begin{align*}
{N}_{33}^h\mathcal{D}_3^h{\bf f} := \left(\begin{array}{c}
({N}_{33}^h\mathcal{D}_3^h{\bf f})_1 \\
({N}_{33}^h\mathcal{D}_3^h{\bf f})_2 \\
({N}_{33}^h\mathcal{D}_3^h{\bf f})_3 
\end{array}\right), \quad ({N}_{33}^h\mathcal{D}_3^h{\bf f})_p = \sum_{q = 1}^{3} N_{33}^h(p,q)\mathcal{D}_3^h{f}^{(q)},\quad p = 1,2,3
\end{align*}
with $\mathcal{D}_3^h{f}^{(q)}$ to be the SBP operator for first spatial derivative in direction $r^{(3)}$ defined as in the first equation of (\ref{sbp_1st_2}). And $N_{3l}^hD_l^h{\bf f}, l=1,2$ are defined similar as those in $\wt{\mathcal{A}}_3^{2h}{\bf c}$.

\section{Bilinear form}\label{appendix_bf}
The term $\mathcal{S}_{2h}({\bf c}_t,{\bf c})$ in \eqref{coarse_simple} is given by
\begin{multline*}
\mathcal{S}_{2h}({\bf c}_t,{\bf c}) = ({D}_1^{2h}{\bf c}_t,{N}_{11}^{2h}{D}_1^{2h}{\bf c})_{2h} +  ({D}_1^{2h}{\bf c}_t,{N}_{12}^{2h}{D}_2^{2h}{\bf c})_{2h} +  ({D}_1^{2h}{\bf c}_t,{N}_{13}^{2h}{D}_3^{2h}{\bf c})_{2h}\\
+  ({D}_2^{2h}{\bf c}_t,{N}_{21}^{2h}{D}_1^{2h}{\bf c})_{2h} 
+  ({D}_2^{2h}{\bf c}_t,{N}_{22}^{2h}{D}_2^{2h}{\bf c})_{2h} +  ({D}_2^{2h}{\bf c}_t,{N}_{23}^{2h}{D}_3^{2h}{\bf c})_{2h} \\
+  ({D}_3^{2h}{\bf c}_t,{N}_{31}^{2h}{D}_1^{2h}{\bf c})_{2h} 
+  ({D}_3^{2h}{\bf c}_t,{N}_{32}^{2h}{D}_2^{2h}{\bf c})_{2h} +  ({D}_3^{2h}{\bf c}_t,{N}_{33}^{2h}{D}_3^{2h}{\bf c})_{2h}\\
+ ({\bf c}_t, P^{2h}_1(N_{11}^{2h}){\bf c})_{2hr} + ({\bf c}_t, P^{2h}_2(N_{22}^{2h}){\bf c})_{2hr} + ({\bf c}_t, P^{2h}_3(N_{33}^{2h}){\bf c})_{2hr},
\end{multline*}
where $P_3^{2h}(N_{33}^{2h})$ is a positive semi-definite operator defined in \eqref{sbp_2nd_1} for direction $r^{(3)}$; $P_1^{2h}(N_{11}^{2h}), P_2^{2h}(N_{22}^{2h})$ are analogue to $P_3^{2h}(N_{33}^{2h})$. 

The term $\mathcal{S}_{h}({\bf f}_t,{\bf f})$ is defined as
\begin{multline*}
\mathcal{S}_{h}({\bf f}_t,{\bf f}) = ({D}_1^{h}{\bf f}_t,{N}_{11}^{h}{D}_1^h{\bf f})_{h} +  ({D}_1^{h}{\bf f}_t,{N}_{12}^{h}{D}_2^h{\bf f})_{h} +  ({D}_1^{h}{\bf f}_t,{N}_{13}^{h}{D}_3^h{\bf f})_{h}\\
+  ({D}_2^{h}{\bf f}_t,{N}_{21}^{h}{D}_1^h{\bf f})_{h} 
+  ({D}_2^{h}{\bf f}_t,{N}_{22}^{h}{D}_2^h{\bf f})_{h} +  ({D}_2^{h}{\bf f}_t,{N}_{23}^{h}{D}_3^h{\bf f})_{h} \\
+  ({D}_3^{h}{\bf f}_t,{N}_{31}^{h}{D}_1^h{\bf f})_{h} 
+  ({D}_3^{h}{\bf f}_t,{N}_{32}^{h}{D}_2^h{\bf f})_{h} +  ({D}_3^{h}{\bf f}_t,{N}_{33}^{h}{D}_3^h{\bf f})_{h}\\
+ ({\bf f}_t, P^{h}_1(N_{11}^h){\bf f})_{hr} + ({\bf f}_t, P^{h}_2(N_{22}^h){\bf f})_{hr} + ({\bf f}_t, P^{h}_3(N_{33}^h){\bf f})_{hr}.
\end{multline*}
Here, $P_l^h(N_{ll}^h)$ are defined similar as $P_l^{2h}(N_{ll}^{2h})$ in $\mathcal{S}_{2h}({\bf c}_t,{\bf c})$.

\bibliography{lu}
\bibliographystyle{siamplain}

\end{document}